\documentclass[11pt,a4paper]{article}
\usepackage{epsf, graphicx}
\usepackage{latexsym,amsfonts,amsbsy,amssymb}
\usepackage{amsmath,amsthm}
\usepackage{cite}
\usepackage{subfigure}

\makeatletter

\@addtoreset{equation}{section} \makeatother
\newtheorem{theorem}{Theorem}[section]
\newtheorem{lemma}{Lemma}[section]

\makeatletter \@addtoreset{equation}{section} \makeatother
\begin{document}

\title{\bf Metric tensors for the interpolation error and its gradient in $L^p$ norm}
\date{}
 \author{Hehu Xie\footnote{LSEC,
ICMSEC, Academy of Mathematics and Systems Science, CAS, Beijing
100080, China \ email: hhxie@lsec.cc.ac.cn} \quad Xiaobo
Yin\footnote{Department of Mathematics, Central China Normal
University, Wuhan 430079, China \ email: yinxb@lsec.cc.ac.cn}}
 \maketitle
 \begin{quote}
\begin{small}
{\bf Abstract.}\,\, A uniform strategy to derive metric tensors in
two spatial dimension for interpolation errors and their gradients
in $L^p$ norm is presented. It generates anisotropic adaptive meshes
as quasi-uniform ones in corresponding metric space, with the metric
tensor being computed based on a posteriori error estimates in
different norms. Numerical results show that the corresponding
convergence rates are always optimal.


\end{small}
\end{quote}
\section{Introduction}
Generation of adaptive meshes is now the standard option in most
software packages. Traditionally, isotropic mesh adaptation has
received much attention, where regular mesh elements are only
adjusted in size based on an error estimate. However, for problems
with anisotropic solutions (with, say, sharp boundary or internal
layers), the shape of elements can be further optimized and an
equidistribution of a scalar error density is not sufficient to
ensure that a mesh is optimally efficient \cite{Azevedo1991}. Indeed
anisotropic meshes have been used successfully in many areas, for
example in singular perturbation and flow problems
\cite{Ait-Ali-Yahia,ApelLube,Becker,HabDomBourAitForVal,HabForDomValBou,PerVahMorZ,ZWu}
and in adaptive procedures
\cite{AgouzalLipVass,Borouchaki1,Borouchaki2,BuDa,CasHecMohPir,Hecht,PerVahMorZ,Rachowicz}.
For anisotropic mesh adaptation, the common practice is to generate
the needed anisotropic mesh as a quasi-uniform one in the metric
space determined by a tensor (or a matrix-valued function), always
called monitor function or metric tensor. Both the monitor function
(denoted by the letter $M$) and metric tensor (denoted by the
calligraphy letter $\mathcal{M}$) play the same role in mesh
generation, i.e., they are used to specify the size, shape, and
orientation of mesh elements throughout the physical domain. The
only difference lies in the way they specify the size of elements.
Indeed, the former specifies the element size through the
equidistribution condition, while the latter determines the element
size through the unitary volume requirement. Readers could regard
the metric tensor as normalization for the monitor function.
Examples of anisotropic meshing strategies include blue refinement
\cite{Kornhuber,Lang}, directional refinement \cite{Rachowicz},
Delaunay-type triangulation method
\cite{Borouchaki1,Borouchaki2,CasHecMohPir,PerVahMorZ}, advancing
front method \cite{Garimella}, bubble packing method
\cite{Yamakawa}, local refinement and modification
\cite{HabDomBourAitForVal,Remacle}, variational methods
\cite{Brackbill,Dvinsky,HuangSiam,Jacquotte,Knupp,LiTangZhang}, and
so on. Readers are referred to \cite{FreyGeorge} and \cite{Owen} for
an overview.

Among these meshing strategies, the definition of the metric tensor
(or monitor function) based on the Hessian of the solution seems
widespread in the meshing
community\cite{Agou3d,CasHecMohPir,ChenSunXu,Azevedo1989,Azevedo1991,AzevedoSimpson,George,HabDomBourAitForVal,
Hecht,Huang,HuangRussell,HuangSiam,Remacle,Vassilevski}. Especially,
Huang and Russell \cite{HuangRussell} propose the monitor function
\begin{eqnarray}
M=\det{\Big (}I+\frac{1}{\alpha}|H(u)|{\Big
)}^{-\frac{1}{d+p(2-m)}}{\Big\|}I+\frac{1}{\alpha}|H(u)|{\Big
\|}^{\frac{mp}{d+p(2-m)}}{\Big [}I+\frac{1}{\alpha}|H(u)|{\Big ]},
\end{eqnarray}
for the interpolation error in $W^{m,p}$ norm ($m=0,1$, $p\in
[1,+\infty)$), where $d$ stands for the spatial dimension. Set
$\mathcal{H}=I+\frac{1}{\alpha}|H(u)|$, when $d=2$,
\begin{eqnarray}\label{HR}
M_{m,p}=\det(\mathcal{H} )^{-\frac{1}{2+p(2-m)}}\|\mathcal{H}
\|^{\frac{mp}{2+p(2-m)}}\mathcal{H}.
\end{eqnarray}
Separately, it becomes
\begin{eqnarray}\label{HRlp}
M_{0,p}=\det(\mathcal{H})^{-\frac{1}{2(p+1)}}\mathcal{H},
\end{eqnarray}
for the interpolation error in $L^p$ norm and
\begin{eqnarray}\label{HRw1p}
M_{1,p}=\det(\mathcal{H})^{-\frac{1}{p+2}}\|\mathcal{H}\|^{\frac{p}{p+2}}\mathcal{H}.
\end{eqnarray}
for the gradient of interpolation error in $L^p$ norm.

The objective of this paper is to give a unified strategy deriving
metric tensors in two spatial dimension for interpolation error and
its gradients in $L^p$ norm. The development begin with the error
estimates \cite{Nadler} for $L^2$ norm and our recent work
\cite{YinXie} for $H^1$ norm on linear interpolation for quadratic
functions on triangles. These estimates are anisotropic in the sense
that they allow a full control of the shape of elements when used
within a mesh generation strategy. Using the relationship between
different norms, a posterior error estimates for other norms
($W^{m,p}, m=0,1$, $p\neq 2$) can be gained. We will apply these
error estimates to formulate corresponding metric tensors in a
unified way. The procedure is based on two considerations: on the
one hand the anisotropic mesh is generated as a quasi-uniform mesh
in the metric tensor. On the other hand, the anisotropic mesh is
required to minimize the error for a given number of triangles. To
compare with those existing methods, we list our main results using
monitor function style, that is
\begin{eqnarray}\label{OursLp}
M_{0,p}^{n}({\bf
x})=\det(\mathcal{H})^{-\frac{1}{2(p+1)}}\mathcal{H},
\end{eqnarray}
for interpolation errors in $L^p$ norm and
\begin{eqnarray}\label{Oursw1p}
M_{1,p}^{n}({\bf
x})=\det(\mathcal{H})^{-\frac{1}{p+2}}\mbox{tr}(\mathcal{H})^{\frac{p}{p+2}}\mathcal{H},
\end{eqnarray}
for gradient of interpolation errors in $L^p$ norm. To sum up, the
metric tensor can be expressed by
\begin{eqnarray}\label{Ours}
M_{m,p}^{n}({\bf x})=\det(\mathcal{H}
)^{-\frac{1}{2+p(2-m)}}\mbox{tr}(\mathcal{H}
)^{\frac{mp}{2+p(2-m)}}\mathcal{H},
\end{eqnarray}
for the $W^{m,p}$ norm ($m=0,1$, $p\in (0,+\infty]$) of the
interpolation error.

The paper is organized as follows. In Section 2, we describe the
anisotropic error estimates on linear interpolation for quadratic
functions on triangles obtained in our recent work \cite{YinXie}.
The formulation of the monitor function and metric tensor is
developed in Section 3. Numerical results are presented in Section 4
to illustrating our analysis. Finally, conclusions are drawn in
Section 5.

\section{Estimates for interpolation error and its
gradient}\label{estimate_section} As we know, the interpolation
error depends on the solution, the size and shape of the elements in
the mesh. Understanding this relation is crucial for the generating
efficient meshes for the finite element method. In the mesh
generation community, this relation is studied more closely for the
model problem of interpolating quadratic functions. This treatment
yields a reliable and efficient estimator of the interpolation error
for general functions provided a saturation assumption is valid
\cite{AgouzalVassilevski,DorflerNochetto}. For instance, Nadler
\cite{Nadler} derived an exact expression for the $L^2$-norm of the
linear interpolation error in terms of the three sides ${\bf
\ell}_1$, ${\bf \ell}_2$, and ${\bf \ell}_3$ of the triangle $K$,
\begin{eqnarray}\label{Nadler}
\|u-u_I\|^2_{L^2(K)}=\frac{|K|}{180}{\Big[}{\Big(}d_1+d_2+d_3{\Big)}^2+d_1d_2+d_2d_3+d_1d_3{\Big]},
\end{eqnarray}
where $|K|$ is the area of the triangle, $d_i = {\bf \ell}_i\cdot
H{\bf \ell}_i$ with $H$ being the Hessian of $u$. Assuming
$u=\lambda_1x^2+\lambda_2y^2$, D'Azevedo and Simpson
\cite{Azevedo1989} derived the exact formula for the maximum norm of
the interpolation error
\begin{eqnarray}\label{AzeSim}
\|(u-u_I)\|^2_{L^{\infty}(K)}=\frac{D_{12}D_{23}D_{31}}{16\lambda_1\lambda_2|K|^2},
\end{eqnarray}
where $D_{ij}={\bf \ell}_i\cdot \mbox{diag}(\lambda_1,\lambda_2){\bf
\ell}_j$. Based on the geometric interpretation of this formula,
they proved that for a fixed area the optimal triangle, which
produces the smallest maximum interpolation error, is the one
obtained by compressing an equilateral triangle by factors
$\sqrt{\lambda_1}$ and $\sqrt{\lambda_2}$ along the two eigenvectors
of the Hessian of $u$. Furthermore, the optimal incidence for a
given set of interpolation points is the Delaunay triangulation
based on the stretching map (by factors $\sqrt{\lambda_1}$ and
$\sqrt{\lambda_2}$ along the two eigenvector directions) of the grid
points. Rippa \cite{Rippa} showed that the mesh obtained in this way
is also optimal for the $L^p$-norm of the error for any $1\leq
p\leq\infty$.

The element-wise error estimates in the following theorem are
developed in \cite{YinXie} using the theory of interpolation and
proper numerical quadrature formula.

\begin{theorem}\label{Theorem2.1}
Let $u$ be a quadratic function and $u_I$ is the Lagrangian linear
finite element interpolation of $u$. The following relationship
holds:
\begin{eqnarray}\label{error_estimate_H1}
\|\nabla(u-u_I)\|^2_{L^2(K)}=\frac{1}{48|K|}\sum_{i=1}^{3}({\bf
\ell}_{i+1}\cdot H{\bf \ell}_{i+2})^2|{\bf \ell}_i|^2,
\end{eqnarray}
where we prescribe $i+3=i,i-3=i$.
\end{theorem}
To get the a posteriori error estimate of the interpolation error in
$L^p$ and $W^{1,p}$ norms for $p\neq 2$, we need some lemmas below.
\begin{lemma}\label{Dnormequi}
For any $d$ positive numbers $a_1, \cdots, a_d$, the inequalities
\begin{eqnarray}
{\Big (}\sum\limits_{j=1}^{d}a_j^2{\Big )}^\frac{1}{2}\leq {\Big
(}\sum\limits_{j=1}^{d}a_j^p{\Big )}^\frac{1}{p}\leq
d^{\frac{1}{p}-\frac{1}{2}}{\Big (}\sum\limits_{j=1}^{d}a_j^2{\Big
)}^\frac{1}{2},
\end{eqnarray}
and
\begin{eqnarray}
d^{\frac{1}{p}-\frac{1}{2}}{\Big (}\sum\limits_{j=1}^{d}a_j^2{\Big
)}^\frac{1}{2}\leq {\Big (}\sum\limits_{j=1}^{d}a_j^p{\Big
)}^\frac{1}{p}\leq {\Big (}\sum\limits_{j=1}^{d}a_j^2{\Big
)}^\frac{1}{2}
\end{eqnarray}
hold for numbers $0<p<2$ and $p>2$, respectively.
\end{lemma}
\begin{proof}
We just give the proof for the case $0<p<2$, it is similar for the
case $p>2$.

 For any number $0<p<2$,
\begin{eqnarray*}
{\Big (}\sum\limits_{j=1}^{d}a_j^2{\Big )}^\frac{1}{2}\leq {\Big
(}\sum\limits_{j=1}^{d}a_j^p{\Big )}^\frac{1}{p}
\end{eqnarray*}
holds due to the Jensen's inequality. From the generalized
arithmetic-mean geometric-mean inequality, for any positive numbers
$a_1, \cdots, a_d$,
\begin{eqnarray*}
{\Big (}\sum\limits_{j=1}^{d}\frac{1}{d}a_j^p{\Big
)}^\frac{1}{p}\leq {\Big
(}\sum\limits_{j=1}^{d}\frac{1}{d}a_j^2{\Big )}^\frac{1}{2}.
\end{eqnarray*}
Then
\begin{eqnarray*}
{\Big (}\sum\limits_{j=1}^{d}a_j^p{\Big )}^\frac{1}{p}\leq
d^{\frac{1}{p}-\frac{1}{2}}{\Big (}\sum\limits_{j=1}^{d}a_j^2{\Big
)}^\frac{1}{2}.
\end{eqnarray*}
\end{proof}
To sum up, for any $d$ positive numbers $a_1, \cdots, a_d$, the
inequalities
\begin{eqnarray}\label{Dnorm2-p}
\underline{C}_p{\Big (}\sum\limits_{j=1}^{d}a_j^2{\Big
)}^\frac{1}{2}\leq {\Big (}\sum\limits_{j=1}^{d}a_j^p{\Big
)}^\frac{1}{p}\leq \overline{C}_p{\Big
(}\sum\limits_{j=1}^{d}a_j^2{\Big )}^\frac{1}{2}
\end{eqnarray}
holds for any numbers $p>0$, where $\underline{C}_p=1$ for $0<p<2$
and $d^{\frac{1}{p}-\frac{1}{2}}$ for $p>2$,
$\overline{C}_p=d^{\frac{1}{p}-\frac{1}{2}}$ for $0<p<2$ and $1$ for
$p>2$.

\begin{lemma}\label{Cnormequi} \cite{AgouzalVassilevski}
For any $p\in (0,+\infty]$ and any non-negative $v\in P_2(K)$ it
holds
\begin{eqnarray}
C_{1/p}^{-\frac{1}{p}}|K|^{\frac{1}{p}-1}\|v\|_{L^1(K)}\leq
\|v\|_{L^p(K)}\leq C_p|K|^{\frac{1}{p}-1} \|v\|_{L^1(K)}
\end{eqnarray}
with
\begin{equation*}
     \left \{
     \begin{array}{lll}
     C_p=1 &\mbox{if}\,\, 0<p\leq 1,&\\
     C_p=(d+1)(d+2)(d!)^{\frac{1}{p}}{\Big (}\prod\limits_{j=1}^{d}(p+j){\Big )}^{-\frac{1}{p}}
     &\mbox{if}\,\, 1<p<+\infty,&\\
     C_{\infty}=\lim\limits_{p\rightarrow +\infty}C_p=(d+1)(d+2),\\
     C_{1/\infty}=\lim\limits_{p\rightarrow
     +\infty}C_{1/p}=1.
     \end{array}
     \right .
\end{equation*}
\end{lemma}

\subsection{Estimates for interpolation errors in $L^p$ norm}
We consider the error of linear interpolation $e=u-u_I$ for a
quadratic function $u$ on $K$. Since the function $e$ is quadratic
on $K$, we can apply Lemma \ref{Cnormequi} to obtain
\begin{eqnarray}\label{relation1-p}
C_{1/p}^{-1/p} |K|^{\frac{1}{p}-1}\| e\|_{L^1(K)}\leq \|
e\|_{L^{p}(K)}\leq C_{p} |K|^{\frac{1}{p}-1}\|e\|_{L^1(K)}.
\end{eqnarray}
Set $p=2$,
\begin{eqnarray*}
|K|^{-\frac{1}{2}}\| e\|_{L^1(K)}\leq \| e\|_{L^{2}(K)}\leq C_{2}
|K|^{-\frac{1}{2}}\|e\|_{L^1(K)},
\end{eqnarray*}
or
\begin{eqnarray}\label{relation2-1}
C_2^{-1}|K|^{\frac{1}{2}}\| e\|_{L^2(K)}\leq \| e\|_{L^{1}(K)}\leq
|K|^{\frac{1}{2}}\|e\|_{L^2(K)}.
\end{eqnarray}
Combine (\ref{relation1-p}) and (\ref{relation2-1}), we get
\begin{eqnarray}\label{relation2-p}
C_{1/p}^{-1/p}C_2^{-1} |K|^{\frac{1}{p}-\frac{1}{2}}\|
e\|_{L^2(K)}\leq \| e\|_{L^{p}(K)}\leq C_{p}
|K|^{\frac{1}{p}-\frac{1}{2}}\|e\|_{L^2(K)}.
\end{eqnarray}
In this article, $A\sim B$ stands for that there exist two constants
$\underline{C}$ and $\overline{C}$ such that
\begin{eqnarray*}
\underline{C}A\leq B\leq \overline{C}A,
\end{eqnarray*}
where the two constants $\underline{C}$ and $\overline{C}$ may
depend on the prescribed error, the index $p$, the dimension $d$,
and the numbers of elements $N$, however are independent of function
at hand. So (\ref{relation2-p}) can be rewritten as
\begin{eqnarray*}
\|e\|_{L^{p}(K)}\sim |K|^{\frac{1}{p}-\frac{1}{2}}\|e\|_{L^2(K)}.
\end{eqnarray*}
Together with the expression (\ref{Nadler}) for the $L^2$ norm of
the linear interpolation error derived by Nadler\cite{Nadler}, we
have the a posteriori error estimate in $L^p$ norms as follows:
\begin{eqnarray}\label{error_estimate_Lp}
\|e\|^2_{L^{p}(K)}\sim |K|^{\frac{2}{p}-1}\|e\|^2_{L^2(K)}
&=&\frac{|K|^{\frac{2}{p}}}{180}{\Big
[}{\Big(}\sum_{i=1}^{3}d_i{\Big)}^2+d_1d_2+d_2d_3+d_1d_3{\Big]}.
\end{eqnarray}

\subsection{Estimates for gradient of interpolation errors in $L^p$ norm}
Now we consider the gradient of linear interpolation error $\nabla
e=\nabla(u-u_I)$ for a quadratic function $u$. Since the function
\begin{eqnarray*}
v_j(x)={\Big (}\frac{\partial e}{\partial x_j}{\Big )}^2
\end{eqnarray*}
is quadratic on $K$, we can apply Lemma \ref{Cnormequi} to obtain
\begin{eqnarray}\label{p-2left}
\|v_j\|_{L^{p/2}(K)}^{1/2}\geq C_{2/p}^{-1/p}
|K|^{\frac{1}{p}-\frac{1}{2}} \|v_j\|_{L^1(K)}^{1/2}=C_{2/p}^{-1/p}
|K|^{\frac{1}{p}-\frac{1}{2}} {\Big \|}\frac{\partial e}{\partial
x_j}{\Big \|}_{L^2(K)},
\end{eqnarray}
and
\begin{eqnarray}\label{p-2right}
\|v_j\|_{L^{p/2}(K)}^{1/2}\leq C_{p/2}^{1/2}
|K|^{\frac{1}{p}-\frac{1}{2}} \|v_j\|_{L^1(K)}^{1/2}=C_{p/2}^{1/2}
|K|^{\frac{1}{p}-\frac{1}{2}} {\Big \|}\frac{\partial e}{\partial
x_j}{\Big \|}_{L^2(K)}.
\end{eqnarray}
Since
\begin{eqnarray*}
\|\nabla e\|_{L^{p}(K)}^{p}=\sum\limits_{j=1}^{d}\int_{K}{\Big
|}\frac{\partial e}{\partial x_j}{\Big |}^pd{\bf
x}=\sum\limits_{j=1}^{d}\|v_j\|_{L^{p/2}(K)}^{p/2},
\end{eqnarray*}
then together with (\ref{p-2left}) and (\ref{p-2right}), we have
\begin{eqnarray*}
C_{2/p}^{-1/p} |K|^{\frac{1}{p}-\frac{1}{2}}{\Big
(}\sum\limits_{j=1}^{d}{\Big \|}\frac{\partial e}{\partial x_j}{\Big
\|}_{L^2(K)}^p{\Big )}^{\frac{1}{p}}\leq \|\nabla e\|_{L^{p}(K)}\leq
C_{p/2}^{1/2} |K|^{\frac{1}{p}-\frac{1}{2}} {\Big
(}\sum\limits_{j=1}^{d}{\Big \|}\frac{\partial e}{\partial x_j}{\Big
\|}_{L^2(K)}^p{\Big )}^{\frac{1}{p}}.
\end{eqnarray*}
From (\ref{Dnorm2-p}), the inequality
\begin{eqnarray*}
\underline{C}_pC_{2/p}^{-1/p} |K|^{\frac{1}{p}-\frac{1}{2}}\|\nabla
e\|_{L^2(K)}\leq \|\nabla e\|_{L^{p}(K)}\leq
\overline{C}_pC_{p/2}^{1/2} |K|^{\frac{1}{p}-\frac{1}{2}}\|\nabla
e\|_{L^2(K)},
\end{eqnarray*}
holds, or simply
\begin{eqnarray*}
\|\nabla e\|_{L^{p}(K)}\sim |K|^{\frac{1}{p}-\frac{1}{2}}\|\nabla
e\|_{L^2(K)}.
\end{eqnarray*}
Together with the a posteriori error estimate
(\ref{error_estimate_H1}) of the interpolation error in
$H^1$($=W^{1,2}$) norm, we have the a posteriori error estimate in
$W^{1,p}$ norms as follows:
\begin{eqnarray}\label{error_estimate_W1p}
\|\nabla e\|^2_{L^p(K)}&\sim& |K|^{\frac{2}{p}-1}\|\nabla
e\|_{L^2(K)}^2\nonumber\\
&=&|K|^{\frac{2}{p}-1}\frac{1}{48|K|}\sum_{i=1}^{3}({\bf
\ell}_{i+1}\cdot H_K{\bf \ell}_{i+2})^2|{\bf
\ell}_i|^2\nonumber\\
&=&\frac{|K|^{\frac{2}{p}-2}}{48}\sum_{i=1}^{3}({\bf
\ell}_{i+1}\cdot H_K{\bf \ell}_{i+2})^2|{\bf \ell}_i|^2.
\end{eqnarray}

\section{Metric tensors for anisotropic mesh adaptation}
We now use the results of Section \ref{estimate_section} to develop
metric tensors for interpolation errors and their gradients in $L^p$
norm in a unified way. As a common practice in anisotropic mesh
generation, the metric tensor, $\mathcal{M}({\bf x})$, is used in a
meshing strategy in such a way that an anisotropic mesh is generated
as a quasi-uniform mesh in the metric space determined by
$\mathcal{M}({\bf x})$. Mathematically, this can be interpreted as
the shape, size and equidistribution requirements as follows.

{\bf The shape requirement.} The elements of the new mesh,
$\mathcal{T}_h$, are (or are close to being) equilateral in the
metric.

{\bf The size requirement.} The elements of the new mesh
$\mathcal{T}_h$ have a unitary volume in the metric, i.e.,
\begin{eqnarray}\label{3.1}
\int_K\sqrt{\det(\mathcal{M}({\bf x}))}d{\bf x}=1,\quad\forall K\in
\mathcal{T}_h.
\end{eqnarray}

{\bf The equidistribution requirement. } The anisotropic mesh is
required to minimize the error for a given number of mesh points (or
equidistribute the error on every element).

Notice that to derive the monitor function, we just need the shape
and equidistribution requirements.
\subsection{Metric tensors for gradients of interpolation errors in $L^p$ norm}
\begin{figure}[ht]
  \centering
  \includegraphics[width=12cm]{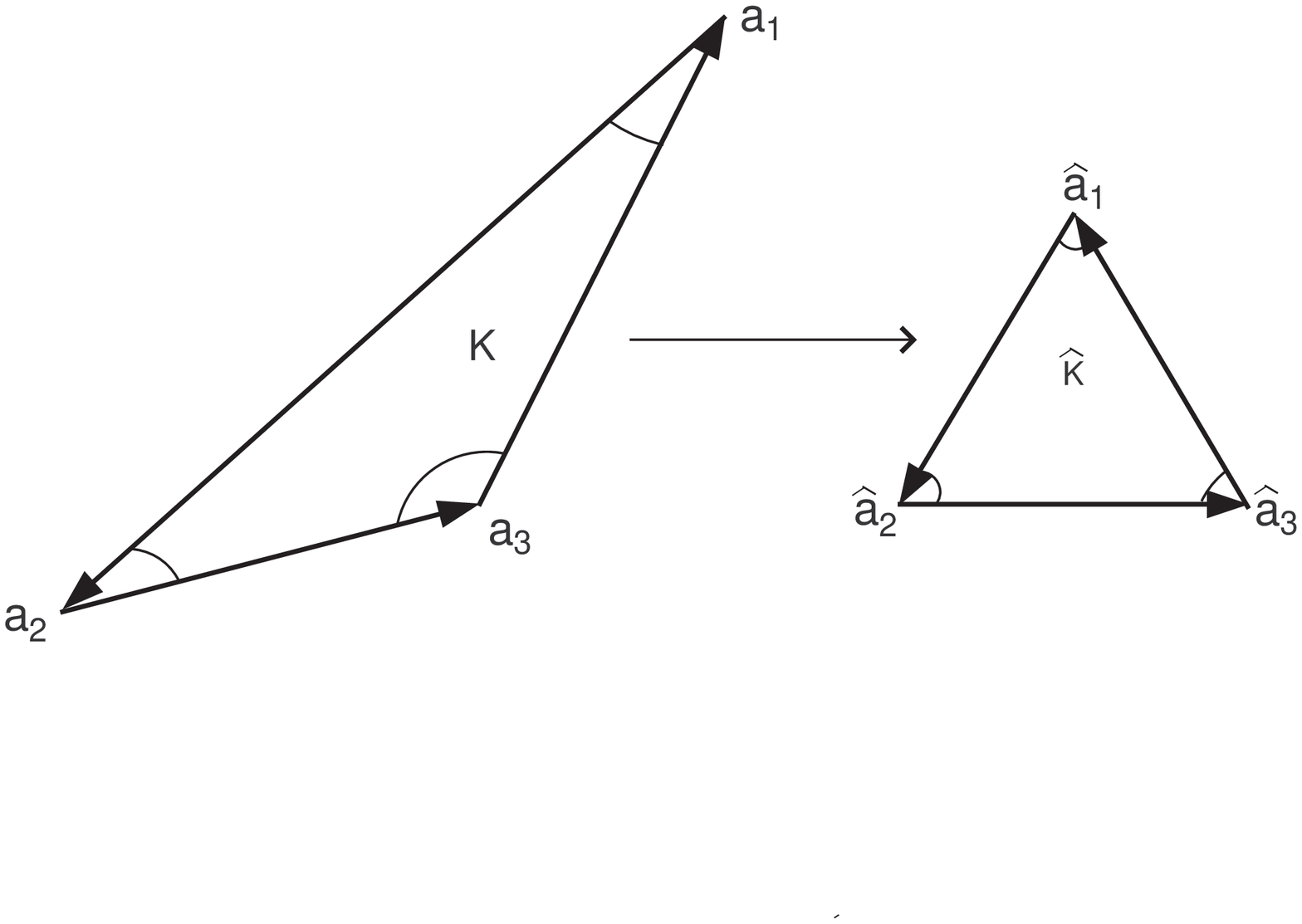}
  \put(-150,181){\large$\mathcal{F}_K$}
  \put(-270,107){\Large$\bf  \ell_1$}
      \put(-178,198){\Large$\bf  \ell_2$}
       \put(-256,183){\Large$\bf  \ell_3$}
  \put(-190,218){\mbox{$\theta_1$}} \put(-292,128){\mbox{$\theta_2$}}
  \put(-240,150){\mbox{$\theta_3$}}
    \put(-73,187){\mbox{$\hat{\theta}_1$}} \put(-100,143){\mbox{$\hat{\theta}_2$}}
  \put(-45,143){\mbox{$\hat{\theta}_3$}}
   \put(-73,123){\mbox{$\hat{\ell}_1$}} \put(-40,165){\mbox{$\hat{\ell}_2$}}
  \put(-98,180){\mbox{$\hat{\ell}_3$}}\vspace{-3.7cm}
  \caption{Affine map ${\bf \hat{x}}=\mathcal{F}_K{\bf x}$ from $K$ to the
reference triangle $\hat{K}$.}\label{Affine_map}
  \end{figure}

\noindent We derive the monitor function $M({\bf x})$ first. Assume
$H(u)$ be a symmetric positive definite matrix on every point ${\bf
x}$, this assumption will be dropped later. Set $M({\bf x})=C({\bf
x})H(u)$. Consider the $L^2$ projection of $H(u)$ on $K$, denoted by
$H_K$, then so does $M_{K}$. Since $H_K$ is a symmetric positive
definite matrix, we consider the singular value decomposition
$H_K=R^T\Lambda R$, where $\Lambda=\mbox{diag}(\lambda_1,\lambda_2)$
is the diagonal matrix of the corresponding eigenvalues
($\lambda_1,\lambda_2>0$) and $R$ is the orthogonal matrix having as
rows the eigenvectors of $H_K$. Denote by $F_K$ and ${\bf t}_K$ the
matrix and the vector defining the invertible affine map $\hat{\bf
x}=\mathcal{F}_K({\bf x})=F_K{\bf x} + {\bf t}_K$ from the generic
element $K$ to the reference triangle $\hat{K}$ (see Figure
\ref{Affine_map}).

Obviously, $M_{K}=C_KH_K$. Let $M_{K}=F_K^TF_K$, then
$F_K=C_K^{\frac{1}{2}}\Lambda^{\frac{1}{2}} R$. Mathematically, the
shape requirement can be expressed as
\begin{eqnarray}\label{shape}
|\hat{\ell}_i|=L\,\,\mbox{and}\,\,
\cos\hat{\theta}_i=\frac{\hat{\ell}_{i+1}\cdot\hat{\ell}_{i+2}}{L^2}=\frac{1}{2},\,i=1,2,3,
\end{eqnarray}
where $L$ is a constant for every element $K$. Enforcing the shape
requirement, we get
\begin{eqnarray*}
\|\nabla e\|^2_{L^p(K)}&\sim&
\frac{|K|^{\frac{2}{p}-2}}{48}\sum_{i=1}^{3}({\bf \ell}_{i+1}\cdot
H_K{\bf \ell}_{i+2})^2|{\bf
\ell}_i|^2\nonumber\\
&=&\frac{|K|^{\frac{2}{p}-2}}{48C_K^2}\sum_{i=1}^{3}({\bf
\ell}_{i+1}\cdot M_K{\bf \ell}_{i+2})^2|{\bf
\ell}_i|^2\nonumber\\
&=&\frac{L^4|K|^{\frac{2}{p}-2}}{48C_K^2}\sum_{i=1}^{3}(\cos\hat{\theta}_i)^2|{\bf
\ell}_i|^2=\frac{L^4|K|^{\frac{2}{p}-2}}{192C_K^2}\sum_{i=1}^{3}|{\bf
\ell}_i|^2.
\end{eqnarray*}
Notice that,
\begin{eqnarray*}
|K|=\frac{|\hat{K}|}{C_K\sqrt{\det(H_K)}},
\end{eqnarray*}
we have
\begin{eqnarray*}
\|\nabla e\|^2_{L^p(K)}&\sim&\frac{L^4|\hat{K}|^{\frac{2}{p}-2}C_K^{2-\frac{2}{p}}\det(H_K)^{1-\frac{1}{p}}}{192C_K^2}\sum_{i=1}^{3}{\Big|}C_K^{-\frac{1}{2}}R^{-1}\Lambda^{-\frac{1}{2}}\hat{\ell}_i{\Big|}^2\nonumber\\
&=&\frac{L^4|\hat{K}|^{\frac{2}{p}-2}\det(H_K)^{1-\frac{1}{p}}}{192C_K^{1+\frac{2}{p}}}\sum_{i=1}^{3}{\Big|}\Lambda^{-\frac{1}{2}}\hat{\ell}_i{\Big|}^2\nonumber\\
&=&\frac{L^4|\hat{K}|^{\frac{2}{p}-2}\det(H_K)^{1-\frac{1}{p}}}{192C_K^{1+\frac{2}{p}}}\frac{\mbox{tr}(H_K)}{\det(H_K)}\\
&=&\frac{L^4|\hat{K}|^{\frac{2}{p}-2}\det(H_K)^{-\frac{1}{p}}\mbox{tr}(H_K)}{192C_K^{1+\frac{2}{p}}}\\
&\sim&\frac{\det(H_K)^{-\frac{1}{p}}\mbox{tr}(H_K)}{C_K^{1+\frac{2}{p}}},
\end{eqnarray*}
then
\begin{eqnarray*}
\|\nabla e\|^p_{L^p(K)}=(\|\nabla e\|^2_{L^p(K)})^{p/2}\sim
\frac{\det(H_K)^{-\frac{1}{2}}\mbox{tr}(H_K)^{\frac{p}{2}}}{C_K^{1+\frac{p}{2}}}.
\end{eqnarray*}
To satisfy the equidistribution requirement, let
\begin{eqnarray*}
\|\nabla e\|^p_{L^p(K)}={\Big (}\sum\limits_{K\in
\mathcal{T}_h}e_K^p{\Big )}/N=\epsilon^p/N,
\end{eqnarray*}
where $N$ is the number of elements of $\mathcal{T}_h$. Then
\begin{eqnarray*}
C_K\sim \det(H_K)^{-\frac{1}{p+2}}\mbox{tr}(H_K)^{\frac{p}{p+2}}.
\end{eqnarray*}
So $M({\bf x})$ could be the form
\begin{eqnarray*}
M({\bf x})=\det(H)^{-\frac{1}{p+2}}\mbox{tr}(H)^{\frac{p}{p+2}}H(u),
\end{eqnarray*}
since $M({\bf x})$ can be modified by multiplying a constant. Since
it corresponds the gradient of interpolation errors in $L^p$ norm,
we denote it by $M_{1,p}^{n}({\bf x})$.

To establish the metric tensor $\mathcal{M}_{1,p}^{n}({\bf x})$, set
$\mathcal{M}_{1,p}^{n}({\bf x})=\theta_{1,p}M_{1,p}^{n}({\bf x})$,
at this time, the size requirement (\ref{3.1}) should be used, which
leads to
\begin{eqnarray*}
\theta_{1,p}\int_K\rho_{1,p}({\bf x}) d{\bf x}=1,
\end{eqnarray*}
where
\begin{eqnarray*}
\rho_{1,p}({\bf x})=\sqrt{\det(M_{1,p}^n({\bf x}))}.
\end{eqnarray*}
Summing the above equation over all the elements of $\mathcal{T}_h$,
one gets
\begin{eqnarray*}
\theta_{1,p}\sigma_{1,p}=N,
\end{eqnarray*}
where
\begin{eqnarray*}
\sigma_{1,p}=\int_{\Omega}\rho_{1,p}({\bf x}) d{\bf x}.
\end{eqnarray*}
Thus, we get
\begin{eqnarray*}
\theta_{1,p}=\frac{N}{\sigma_{1,p}},
\end{eqnarray*}
and as a consequence,
\begin{eqnarray*}
\mathcal{M}_{1,p}^{n}({\bf
x})=\frac{N}{\sigma_{1,p}}\det(H)^{-\frac{1}{p+2}}\mbox{tr}(H)^{\frac{p}{p+2}}H(u).
\end{eqnarray*}

\subsection{Metric tensor for the interpolation errors in $L^p$ norm}
Using the error estimates (\ref{error_estimate_Lp}) for
interpolation errors in $L^p$ norm and the shape requirement
(\ref{shape}), we have
\begin{eqnarray*}
\|e\|^2_{L^{p}(K)}&\sim& \frac{|K|^{\frac{2}{p}}}{180}{\Big [}{\Big(}\sum_{i=1}^{3}d_i{\Big)}^2+d_1d_2+d_2d_3+d_1d_3{\Big]}\nonumber\\
&=&\frac{|K|^{\frac{2}{p}}}{180C_K^2}{\Big[}{\Big(}\sum_{i=1}^{3}|\hat{\ell}_i|^2{\Big)}^2+\sum_{i=1}^{3}{\Big(}|\hat{\ell}_{i+1}||\hat{\ell}_{i+2}|{\Big)}^2{\Big]}\nonumber\\
&=&\frac{L^4|K|^{\frac{2}{p}}}{15C_K^2}=\frac{L^4|\hat{K}|^{\frac{2}{p}}}{15C_K^{2+\frac{2}{p}}\det(H)^{\frac{1}{p}}}
\sim \frac{1}{C_K^{2+\frac{2}{p}}\det(H)^{\frac{1}{p}}}.
\end{eqnarray*}
Then,
\begin{eqnarray*}
\|e\|^p_{L^{p}(K)}=(\|e\|^2_{L^{p}(K)})^{p/2}\sim
\frac{1}{C_K^{p+1}\det(H)^{\frac{1}{2}}}.
\end{eqnarray*}
To satisfy the equidistribution requirement, let
\begin{eqnarray*}
\|e\|^p_{L^p(K)}={\Big (}\sum\limits_{K\in \mathcal{T}_h}e_K^p{\Big
)}/N=\epsilon^p/N.
\end{eqnarray*}
Using similar argument in last subsection, we easily get monitor
functions
\begin{eqnarray*}
M_{0,p}^{n}({\bf x})=\det(H)^{-\frac{1}{2(p+1)}}H(u),
\end{eqnarray*}
and metric tensors
\begin{eqnarray*}
\mathcal{M}_{0,p}^{n}({\bf x})=\frac{N}{\sigma_{0,p}}\det
(H)^{-\frac{1}{2(p+1)}}H(u),
\end{eqnarray*}
for the interpolation errors in $L^p$ norm.
\subsection{Practice use of metric tensor}
So far we assume that $H(u)$ is a symmetric positive definite matrix
at every point. However this assumption doesn't hold in many cases.
In order to obtain a symmetric positive definite matrix, the
following procedure are often implemented. First, the Hessian $H$ is
modified into $|H|=R^T\,\,\mbox{diag}(|\lambda_1|,|\lambda_2|)R$ by
taking the absolute value of its eigenvalues
(\cite{HabForDomValBou}). Since $|H|$ is only semi-positive
definite, $\mathcal{M}_{m,p}^{n}$ cannot be directly applied to
generate the anisotropic meshes. To avoid this difficulty, we
regularize the expression with the flooring parameter
$\alpha_{m,p}>0$ (see, e.g., \cite{HuangSiam}). Replacing $|H|$ with
\begin{eqnarray*}
\mathcal{H}=\alpha_{m,p}I+|H|,
\end{eqnarray*}
we get the modified metric tensors, also denoted by
$\mathcal{M}_{m,p}^{n}$, that is
\begin{eqnarray}\label{Oursmetric}
\mathcal{M}_{m,p}^{n}({\bf
x})=\frac{N}{\sigma_{m,p}}\det(\mathcal{H}
)^{-\frac{1}{2+p(2-m)}}\mbox{tr}(\mathcal{H}
)^{\frac{mp}{2+p(2-m)}}\mathcal{H},
\end{eqnarray}
which are suitable for practical mesh generation.
\subsection{Comparison with existing methods using monitor function style}
When $m=0$, the new monitor function $M_{0,p}^{n}$ (\ref{OursLp}) is
in fact the same with (\ref{HRlp}) in \cite{HuangRussell,Huang}.
Chen, Sun and Xu \cite{ChenSunXu} proved that under suitable
conditions, the error estimate
\begin{eqnarray*}
\|u-u_I\|_{L^p(\Omega)}\leq CN^{-\frac{2}{d}}\|\sqrt[d]{\det
H}\|_{L^{\frac{pd}{2p+d}}(\Omega)}, 1\leq p\leq \infty,
\end{eqnarray*}
holds on the quasi-uniform mesh determined by the metric $(\det
{H})^{-\frac{1}{2p+d}}H$, where $H$ is a majorant of the Hessian
matrix, $N$ is the number of elements in the triangulation and the
constant $C$ does not depend on $u$ and $N$. This estimate is
optimal in the sense that it is a lower bound if $u$ is strictly
convex or concave. Note that $\mathcal{H}$ can be chosen as a
majorant of the Hessian matrix.

When $m=1$, the new monitor function $M_{1,p}^{n}$ (\ref{Oursw1p})
is different with (\ref{HRw1p}) \cite{HuangRussell} that the former
refers to $\mbox{tr}(\mathcal{H})$ and the latter involves
$\|\mathcal{H}\|$. In some cases, the two monitor functions are
pretty much alike. However, in other cases, the effect of the former
is superior to the latter for mesh generation. Numerical results in
\cite{XieYin} have shown our approach's superiority for the error in
$H^{1}$ norm.

\begin{figure}
\hspace{-1cm}
\subfigure[]{\includegraphics[width=8cm]{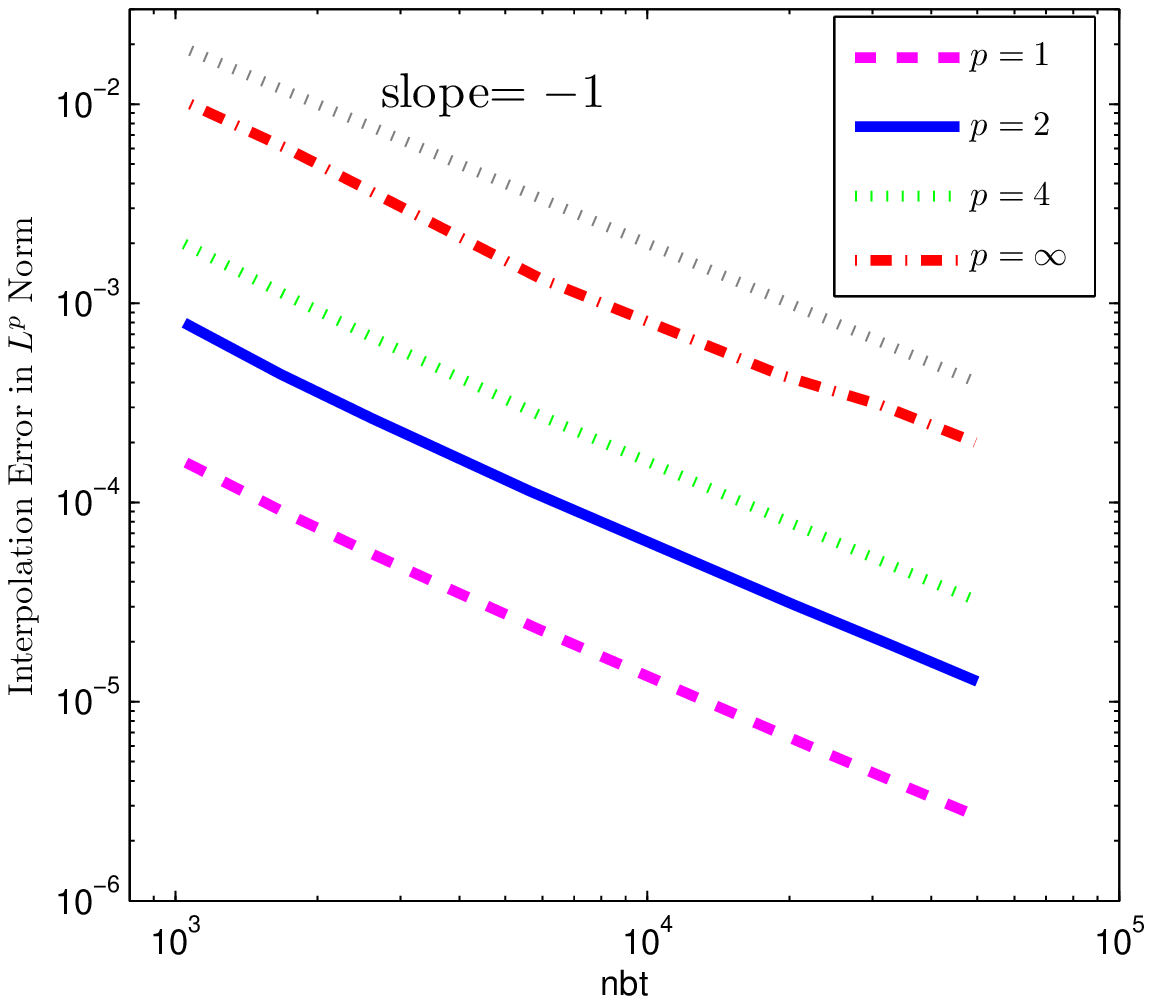}}
\hspace{-1cm}
\subfigure[]{\includegraphics[width=8cm]{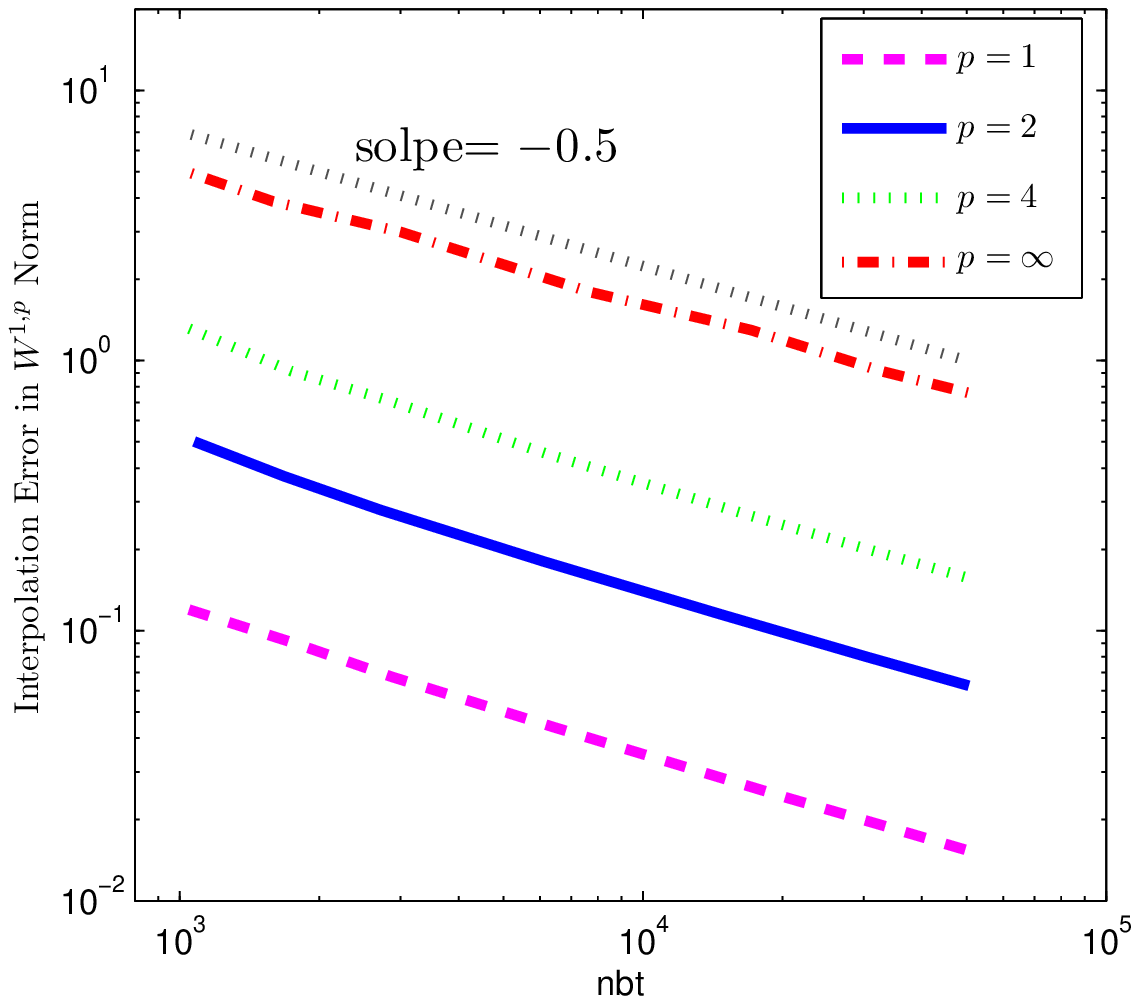}}
\caption{Example 1: Interpolation error and its gradient in $L^p$
norm}\label{Fig_Interperror_Circle}
\end{figure}
\section{Numerical experiments}
In this section, we present some numerical results for three
problems with given analytical solutions. The numerical results are
performed by using the BAMG software \cite{Hecht}. Given a
¡°background¡± mesh and an approximation solution, BAMG generates
the mesh according to the metric tensor. The code allows the user to
supply his/her own metric tensor defined on a background mesh. In
our computation, the background mesh has been taken as the most
recent mesh available.

Denote by $nbt$ the number of triangles in the current mesh. The
number of triangles is adjusted when necessary by trial and errors
through the modification of the multiplicative coefficient of the
metric tensors.

\begin{figure}
\subfigure[]{\includegraphics[width=8cm]{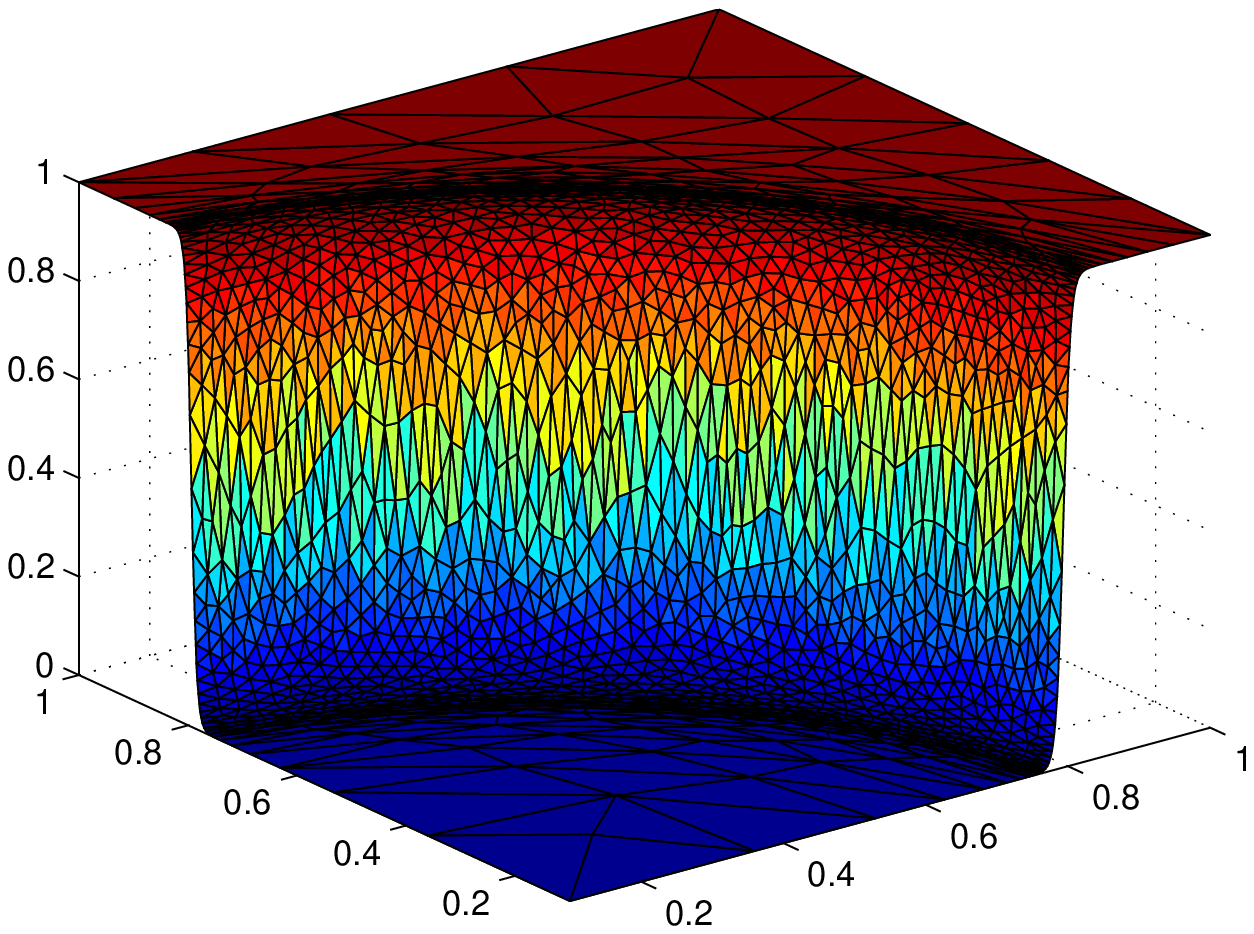}}
\subfigure[]{\includegraphics[width=8cm]{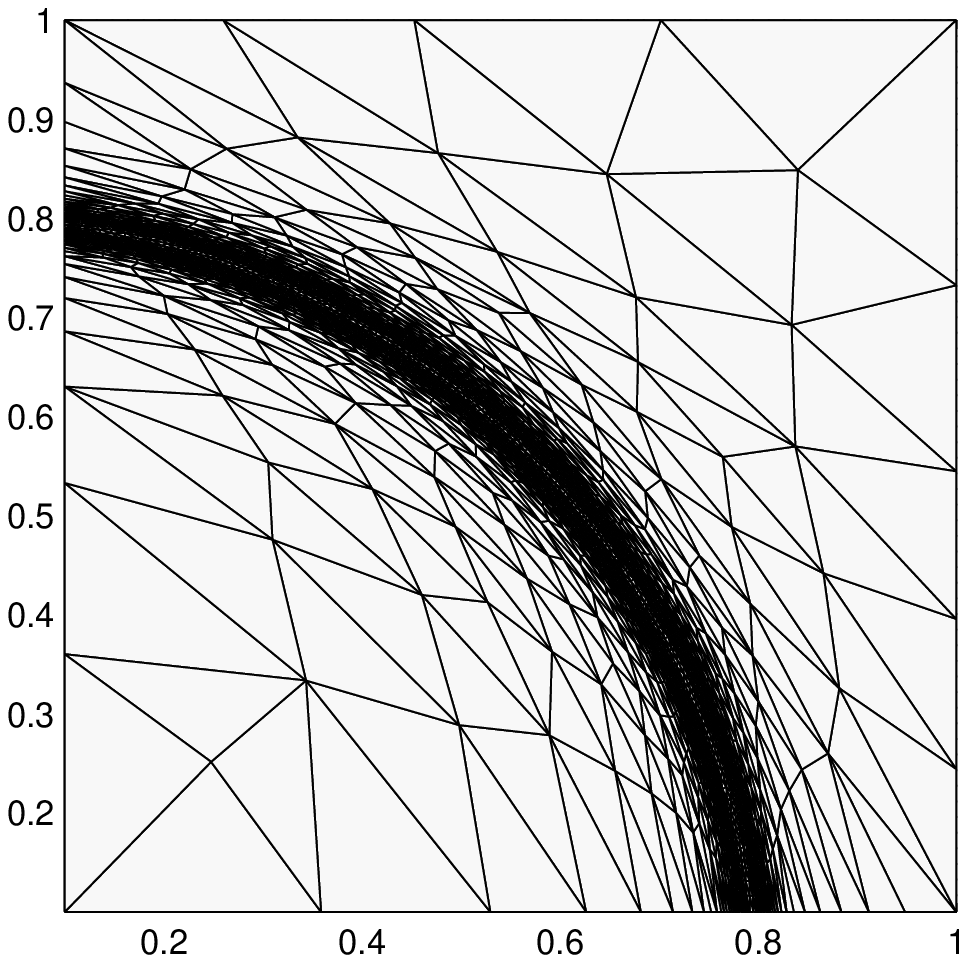}}
\caption{Example 1: plots of the solution (a) and corresponding mesh
(b) using $\mathcal{M}_{1,1}$}\label{Fig_Mesh_Solution_Circle}
\end{figure}

 \noindent{\bf Example 1} This example is to generate adaptive
 meshes for
\begin{eqnarray}
u({\bf x})=\frac{1}{1+e^{-200(\sqrt{x_1^2+x_2^2}-0.8)}},\quad {\bf
x}\in (0.1,1)\times(0.1,1).
\end{eqnarray}
This function is anisotropic along the quarter circle
$x_1^2+x_2^2=0.8^2$ and changes sharply in the direction normal to
this curve. A similar example was presented in \cite{Nguyen} where
the region is $(0,1)\times(0,1)$. In the current computation, each
run is stopped after 15 iterations to guarantee that the adaptive
procedure tends towards stability. We show in Figure
\ref{Fig_Interperror_Circle} the $L^p$ norms of the interpolation
error and its gradient using corresponding metric tensors, for
$p=1,2,4,\infty$. For example, the curve $p=2$ in (a) stands for the
interpolation error using the metric tensor $\mathcal{M}_{0,2}$,
while $p=\infty$ in (b) stands for the gradient interpolation error
using the metric tensor $\mathcal{M}_{1,\infty}$. We see that the
convergence rates for the interpolation error and its gradient are
always nearly optimal, i.e. $\|e\|_{L^p}\sim N^{-1}$ and $\|\nabla
e\|_{L^p}\sim N^{-0.5}$. We also show in Figure
\ref{Fig_Mesh_Solution_Circle} plots of the solution and
corresponding mesh using the metric tensor $\mathcal{M}_{1,1}$.

\begin{figure}
\hspace{-1cm}
\subfigure[]{\includegraphics[width=8cm]{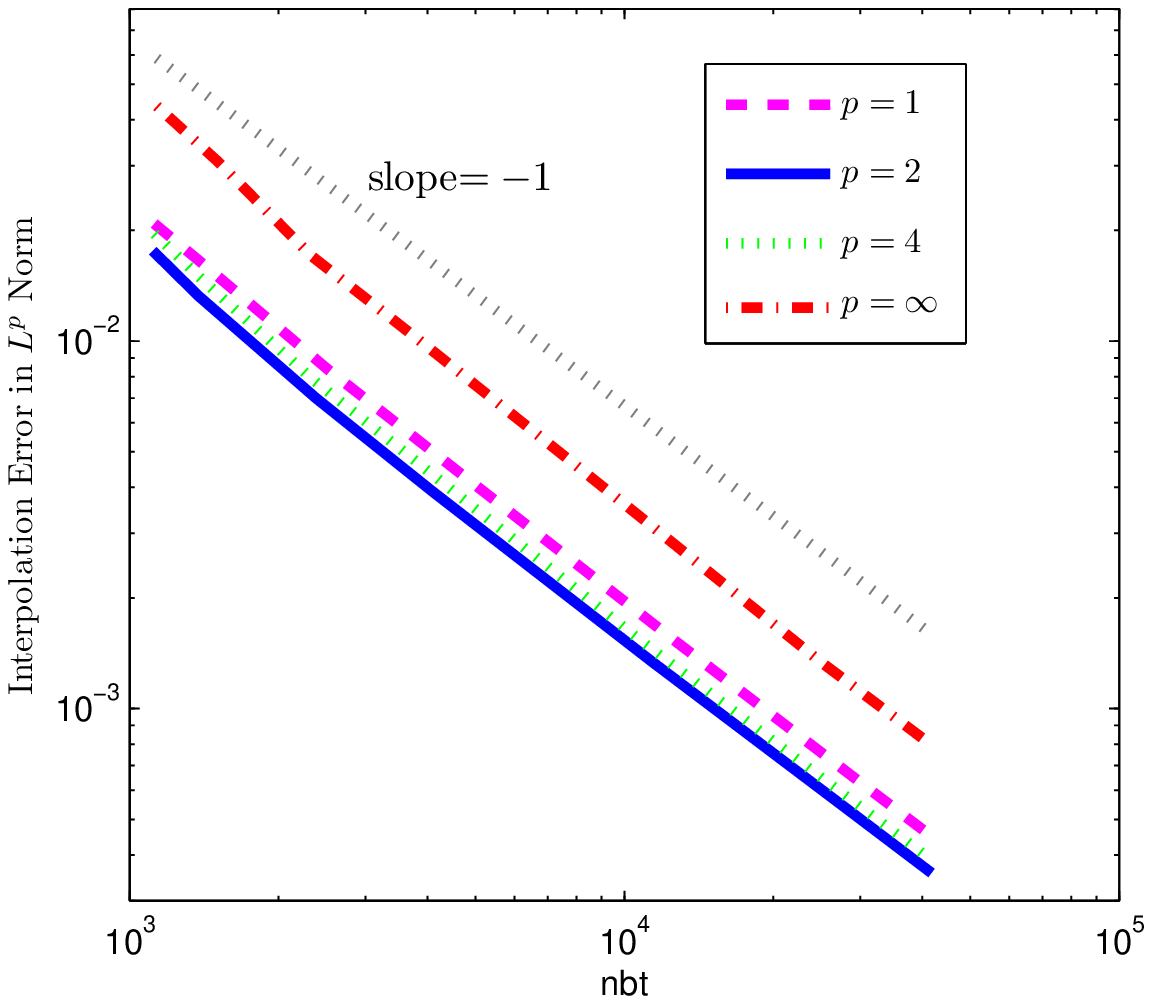}}
\hspace{-1cm}
\subfigure[]{\includegraphics[width=8cm]{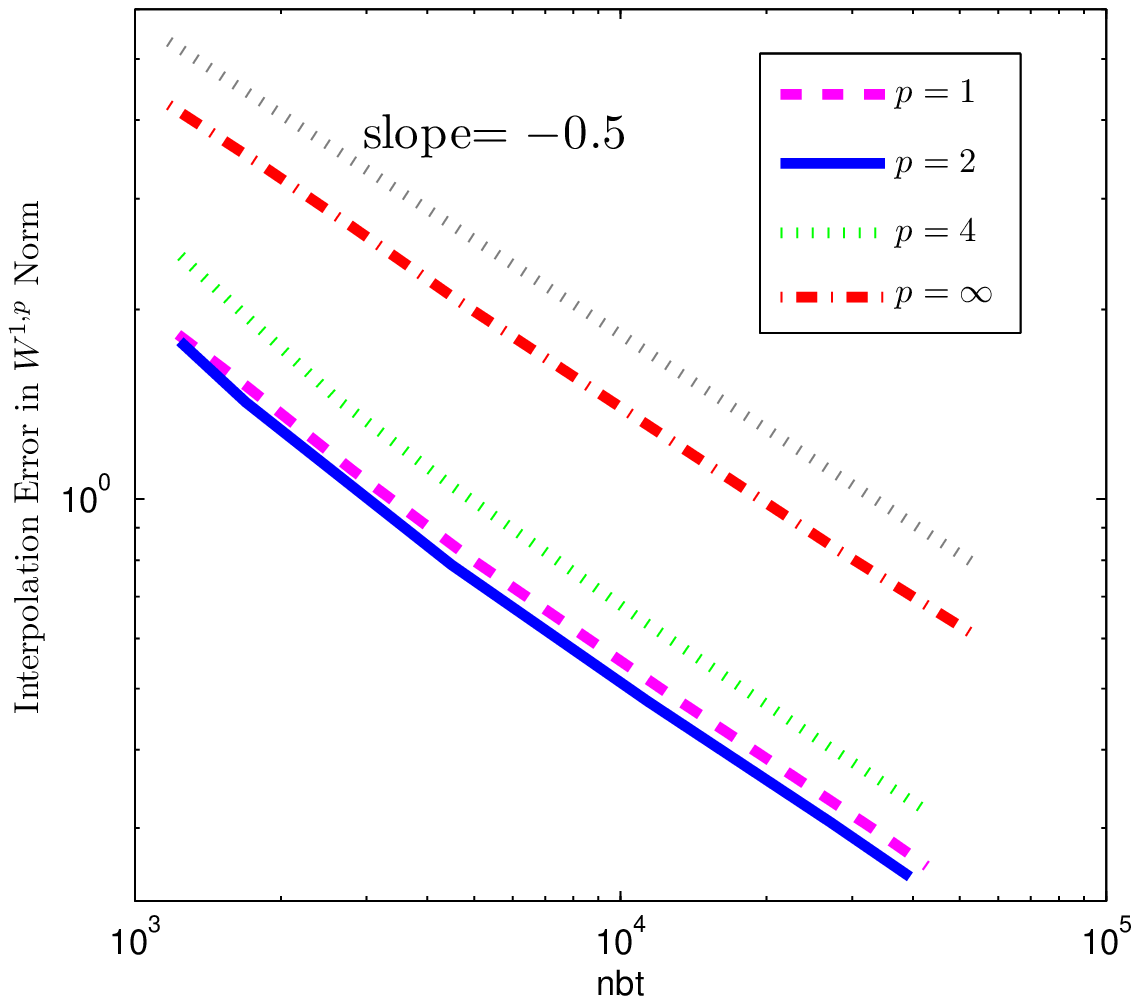}}
\caption{Example 2: Interpolation error and its gradient in $L^p$
norm}\label{Fig_Interp_error_ZigZag}
\end{figure}

 \noindent{\bf Example 2} This example is to generate adaptive
 meshes for
\begin{eqnarray}
u({\bf x})=x_1^2x_2+x_2^3+\tanh(10(\sin(5x_2)-2x_1)),\quad {\bf
x}\in (-1,1)\times(-1,1).
\end{eqnarray}
This function is anisotropic along the zigzag curve
$\sin(5x_2)-2x_1=0$ and changes sharply in the direction normal to
this curve (taken from \cite{AgouzalVassilevski}). In the current
computation, each run is stopped after 20 iterations to guarantee
that the adaptive procedure tends towards stability. We show in
Figure \ref{Fig_Interp_error_ZigZag} the $L^p$ norms of the
interpolation error and its gradient using corresponding metric
tensors, for $p=1,2,4,\infty$. As in Example 1, the convergence
rates for the interpolation error and its gradient here are always
nearly optimal. In Figure \ref{Fig_meshes_ZigZag} we select 6 meshes
with 4000 triangles generated by corresponding metric tensors. We
can learn that the optimal meshes in different norms are different.
For example, the mesh generated by the metric tensor
$\mathcal{M}_{1,\infty}$ concentrates more triangles and nodes along
the zigzag line.

\begin{figure}
\vspace{-1cm} \hspace{-1.1cm}
\subfigure[]{\includegraphics[width=6cm]{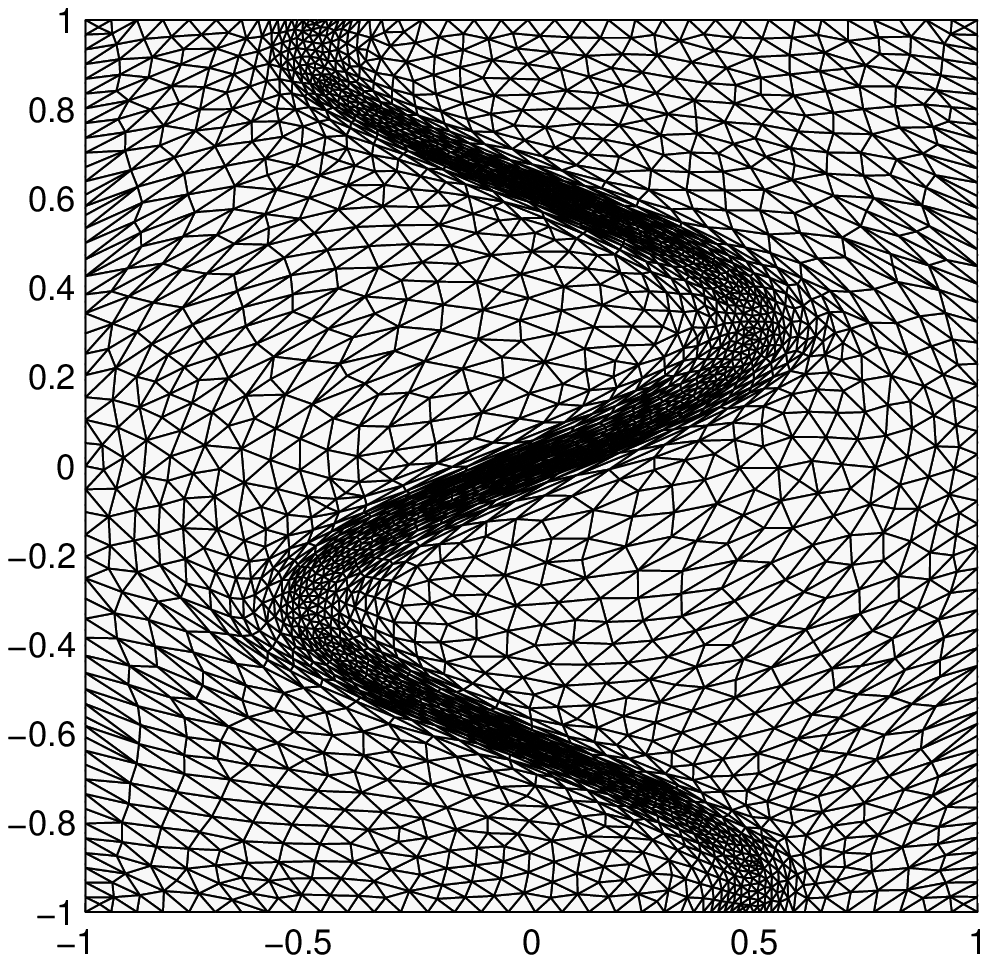}}
\hspace{-1.2cm}
\subfigure[]{\includegraphics[width=6cm]{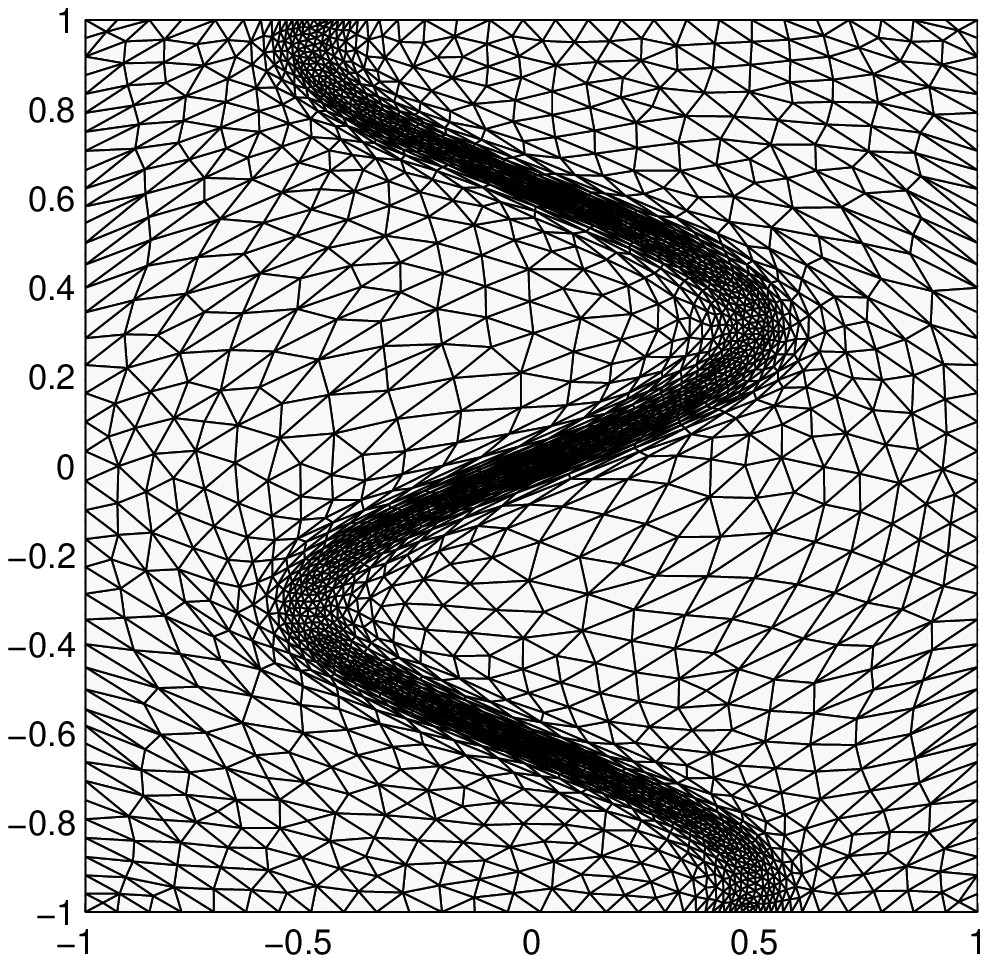}}
\hspace{-1.2cm}
\subfigure[]{\includegraphics[width=6cm]{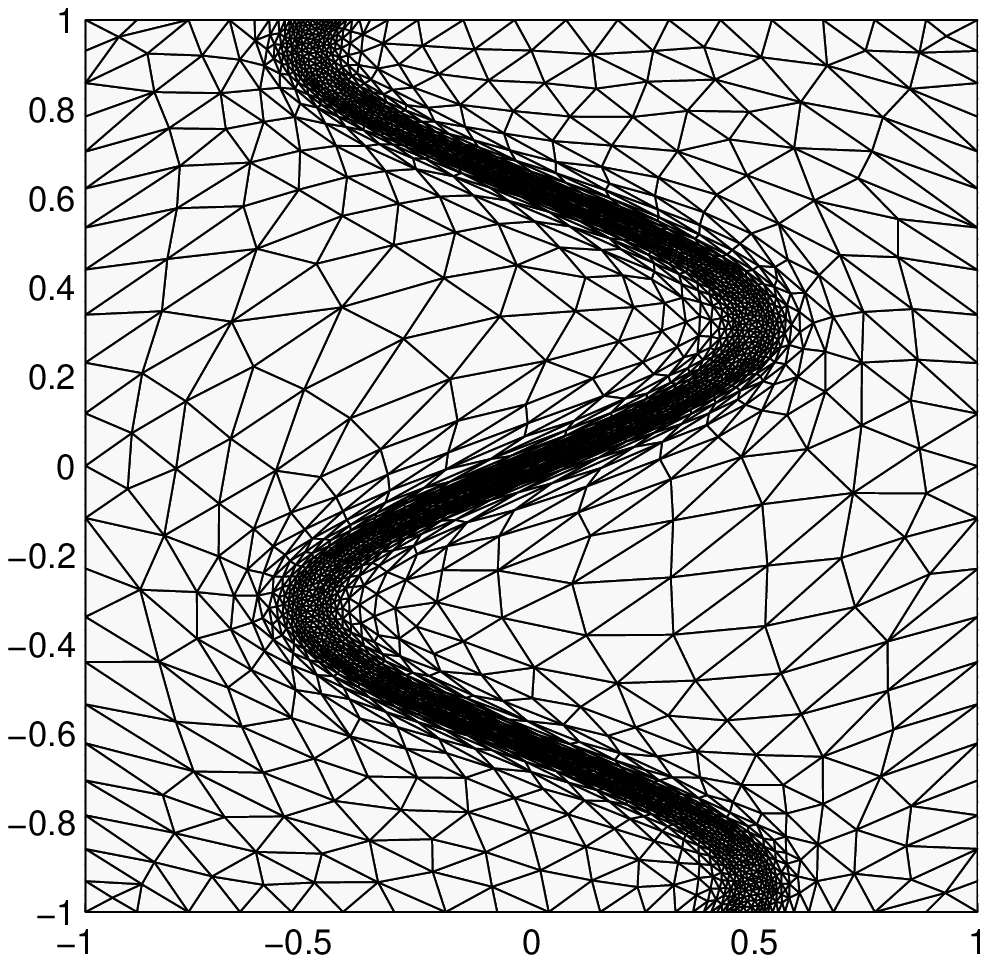}}

\hspace{-1.1cm}
\subfigure[]{\includegraphics[width=6cm]{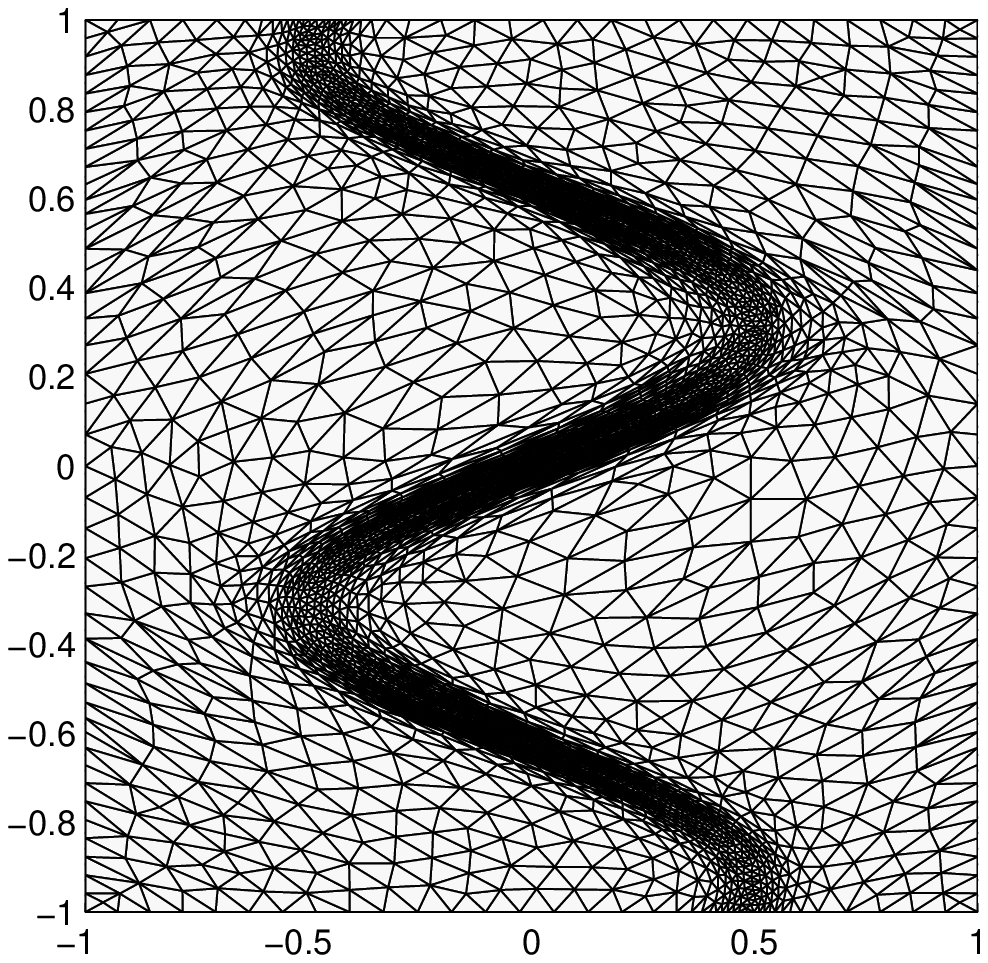}}
\hspace{-1.2cm}
\subfigure[]{\includegraphics[width=6cm]{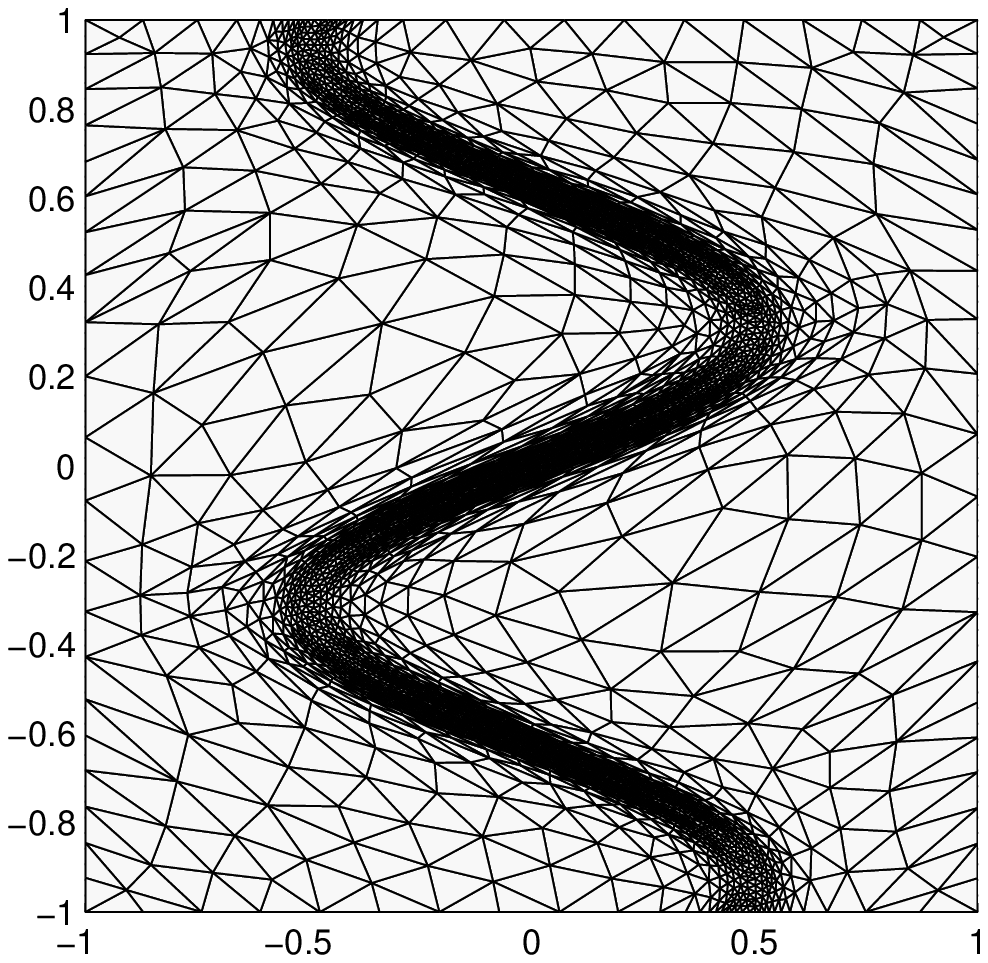}}
\hspace{-1.2cm}
\subfigure[]{\includegraphics[width=6cm]{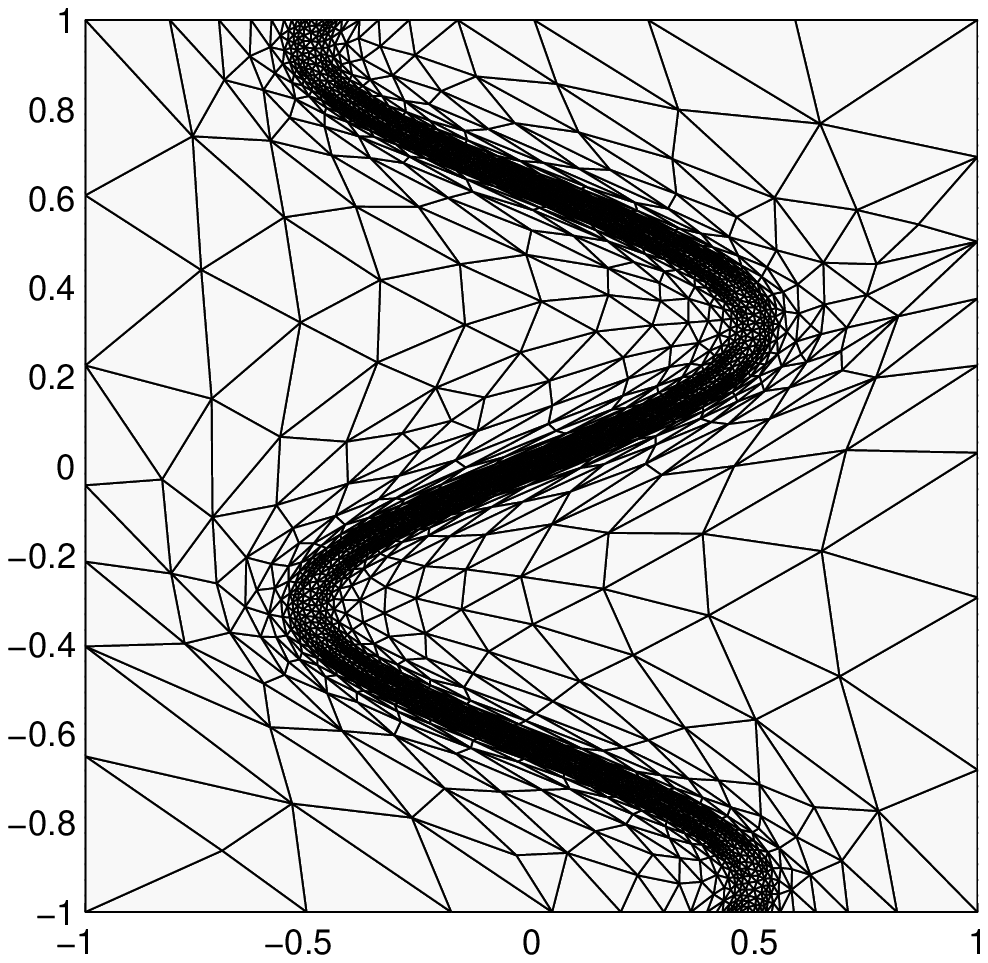}}
\caption{Example 2: Meshes generated by the metric tensor
$\mathcal{M}_{m,p}$ for (a)$m=0,p=1$, (b)$m=0,p=2$,
(c)$m=0,p=\infty$, (d)$m=1,p=1$, (e)$m=1,p=2$,
(f)$m=1,p=\infty$.}\label{Fig_meshes_ZigZag}
\end{figure}

\noindent{\bf Example 3} (Taken from \cite{XieYin}) This example is
to solve the boundary value problem of Poisson's equation
\begin{eqnarray}
-\triangle u&=&f,\quad {\bf x}\in
\Omega\equiv(-1.2,1.2)\times(-1.2,1.2),
\end{eqnarray}
with the Dirichlet boundary condition and the right-hand side term
being chosen such that the exact solution is given by
\begin{eqnarray}
u({\bf x})=\sum_{i=1}^{5}\big[
(1+e^{\frac{x+y-c_i}{2\epsilon}})^{-1}+(1+e^{\frac{x-y-d_i}{2\epsilon}})^{-1}\big],
\end{eqnarray}
where $c_i=0,-0.6,0.6,-1.2,1.2;\, d_i=0,-0.6,0.6,-1.2,1.2.$ The
solution exhibits ten sharp layers on lines $x+y-c_i=0$ and
$x-y-d_i=0$, $i=1,2,\cdots,5$, when $\epsilon$ is small. In our
computations, $\epsilon$ is taken as 0.01. Numerical results in
\cite{XieYin} have shown that our approach's superiority for the
error in $H^{1}$ norm. In the current computation, each run is
stopped after 20 iterations to guarantee that the adaptive procedure
tends towards stability, except that governed by
$\mathcal{M}_{1,\infty}$, which need 30 iterations. We show in
Figure \ref{Fig_Interp_error_Tenlines} the $L^p$ norms of the
interpolation error and its gradient using corresponding metric
tensors, for $p=1,2,4,\infty$. As in Example 1 and Example 2, the
convergence rates for the interpolation error and its gradient here
are always nearly optimal. Another purpose to select this example is
to describe the difference of finding layers using different norms.
In Figure \ref{Fig_meshes_Tenlines_W1p} we list meshes in different
stage during one selected run governed by corresponding metric
tensors. While in Figure \ref{Fig_loops_Tenlines} convergence
history is shown. From the three figures we can learn that most of
the metric tensors can quickly find the layers except the metric
tensor $\mathcal{M}_{1,\infty}$ when dealing with the complex
problems, e.g., with multiple layers.

\begin{figure}
\hspace{-1cm}
\subfigure[]{\includegraphics[width=8cm]{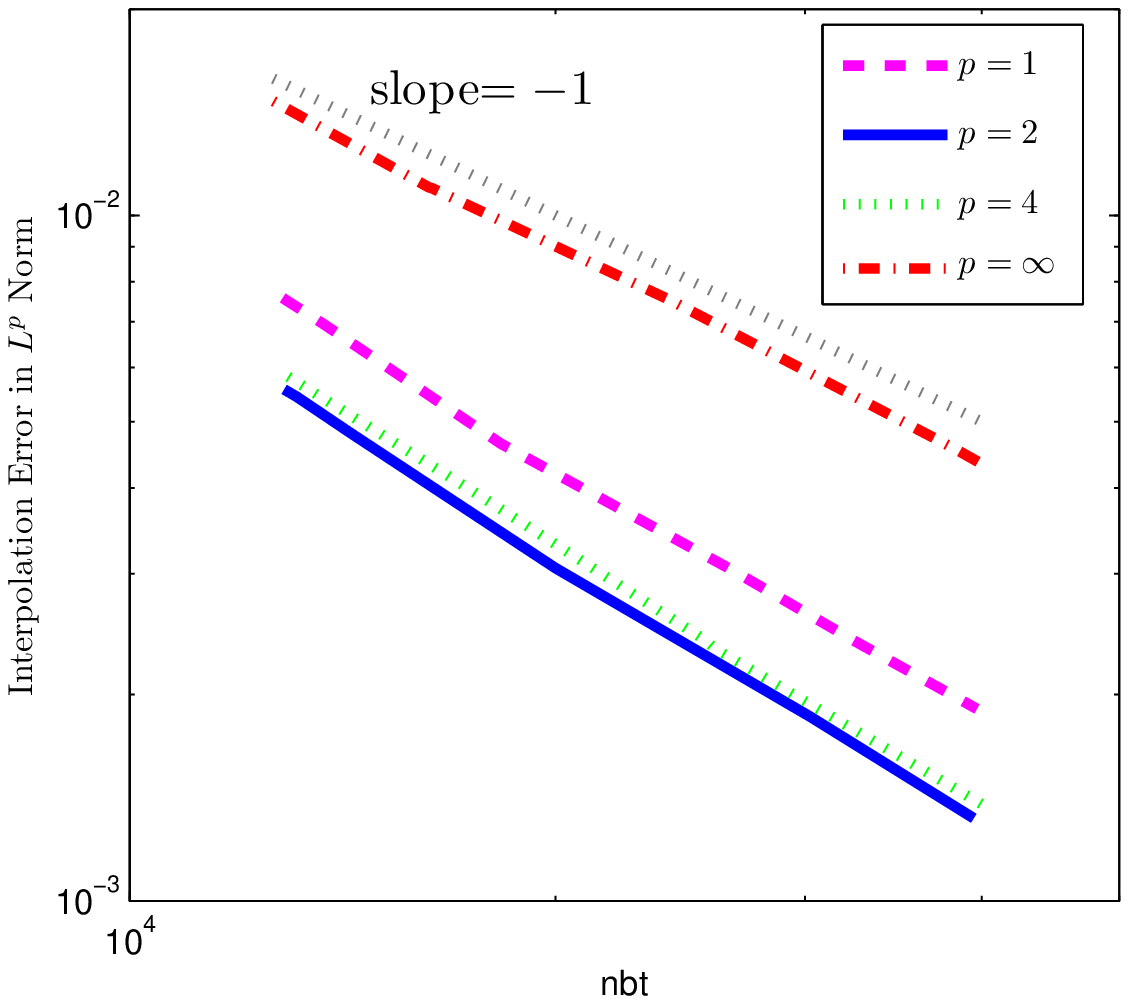}}
\hspace{-1cm}
\subfigure[]{\includegraphics[width=8cm]{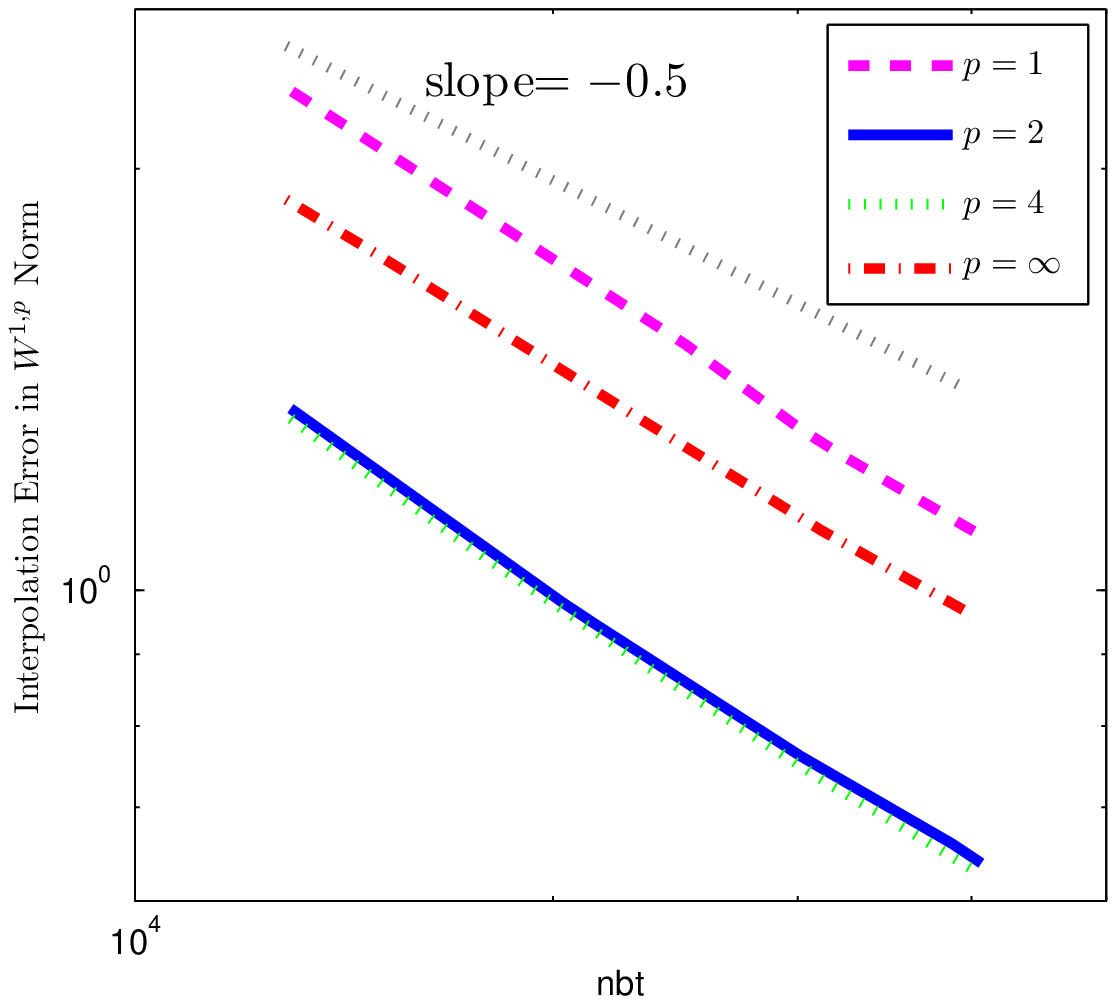}}
\caption{Example 3: Interpolation error and its gradient in $L^p$
norm}\label{Fig_Interp_error_Tenlines}
\end{figure}

\begin{figure}
\vspace{-1cm} \hspace{-1.1cm}
\subfigure[]{\includegraphics[width=6cm]{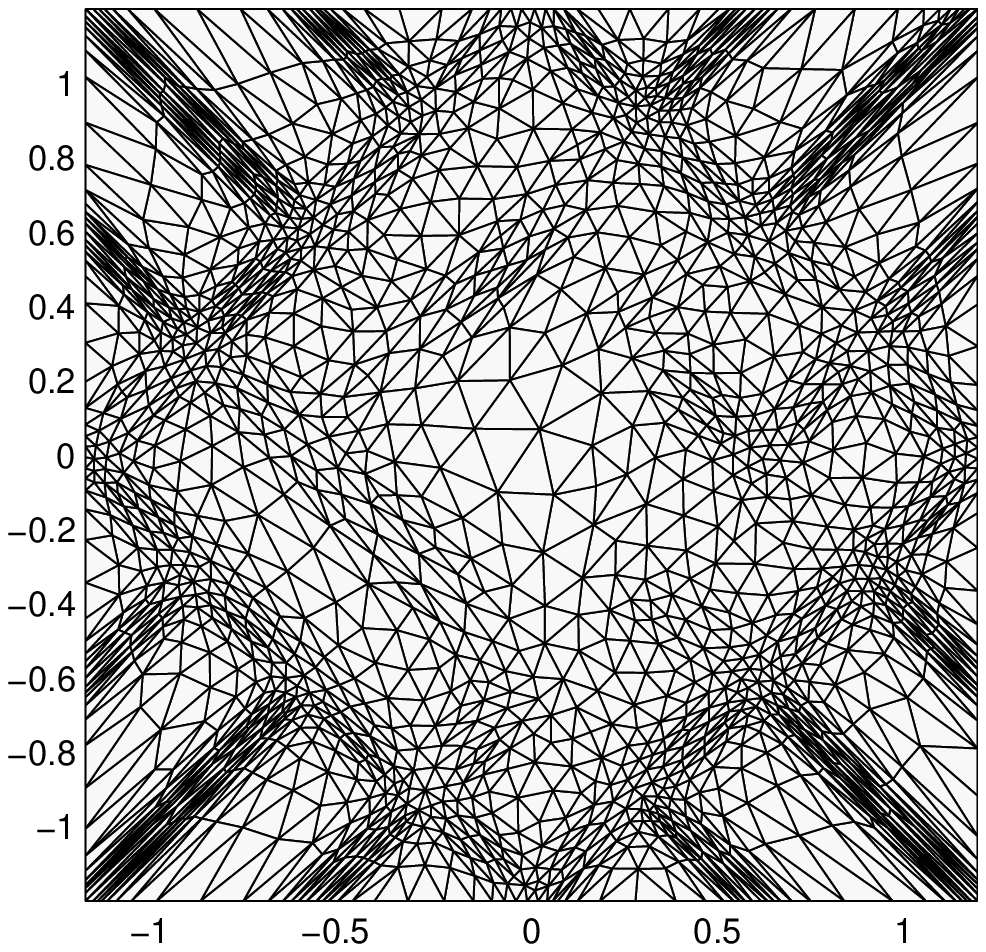}}
\hspace{-1.2cm}
\subfigure[]{\includegraphics[width=6cm]{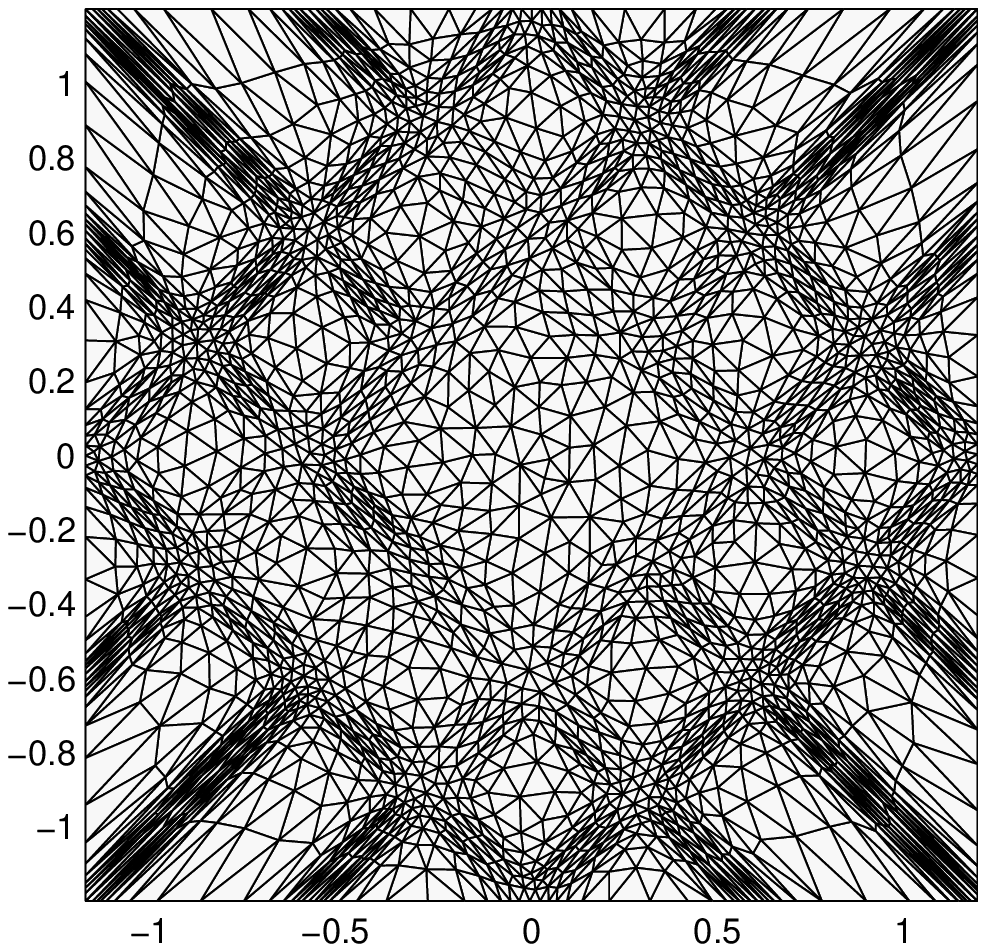}}
\hspace{-1.2cm}
\subfigure[]{\includegraphics[width=6cm]{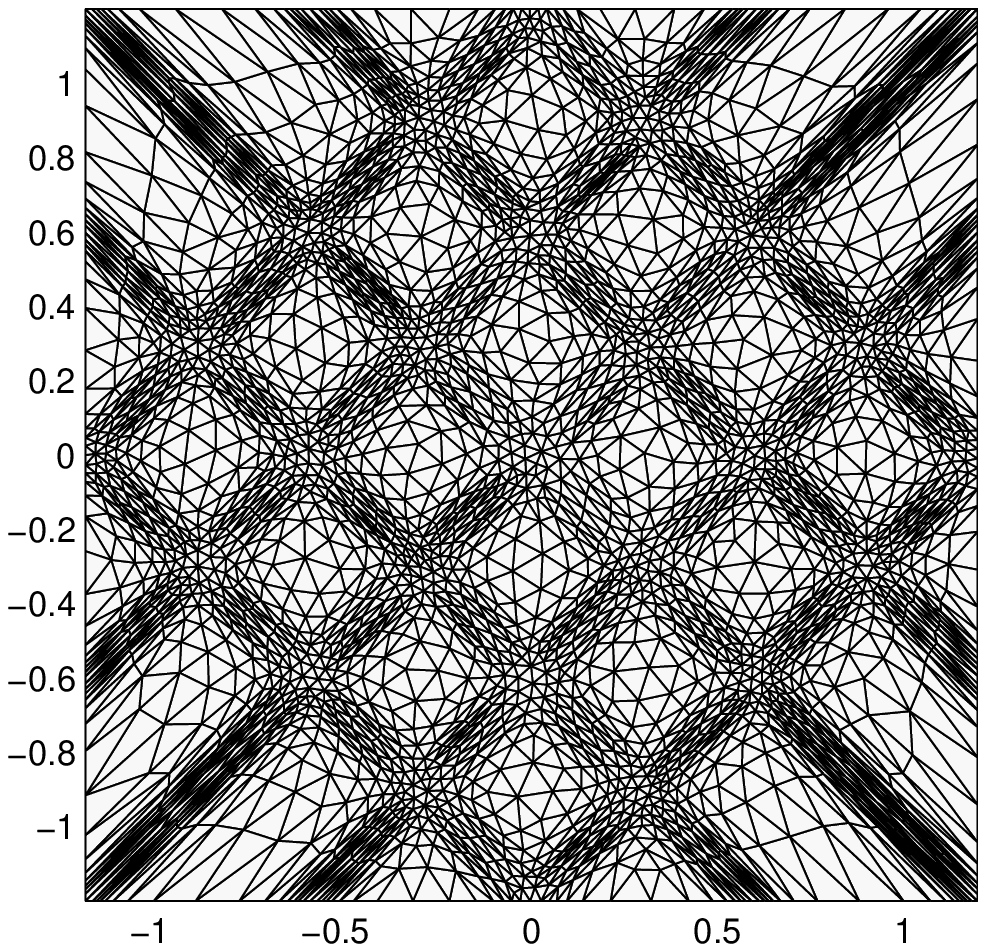}}

\hspace{-1.1cm}
\subfigure[]{\includegraphics[width=6cm]{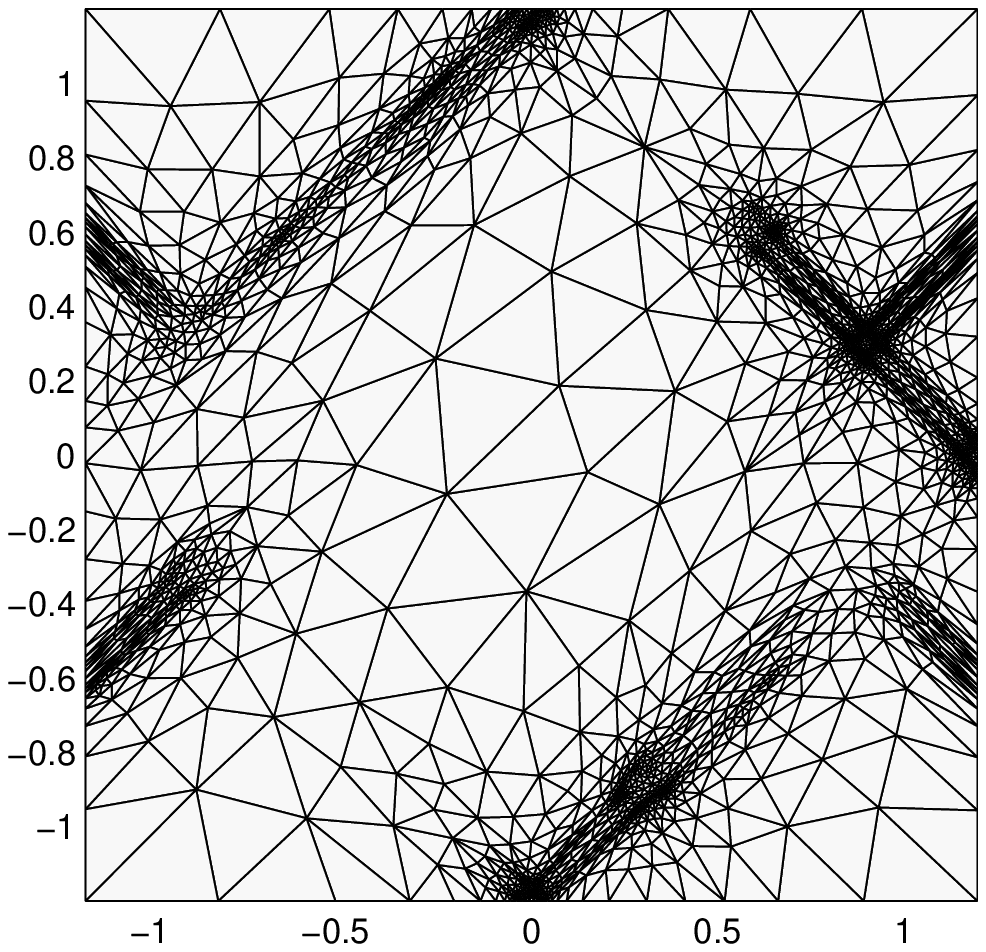}}
\hspace{-1.2cm}
\subfigure[]{\includegraphics[width=6cm]{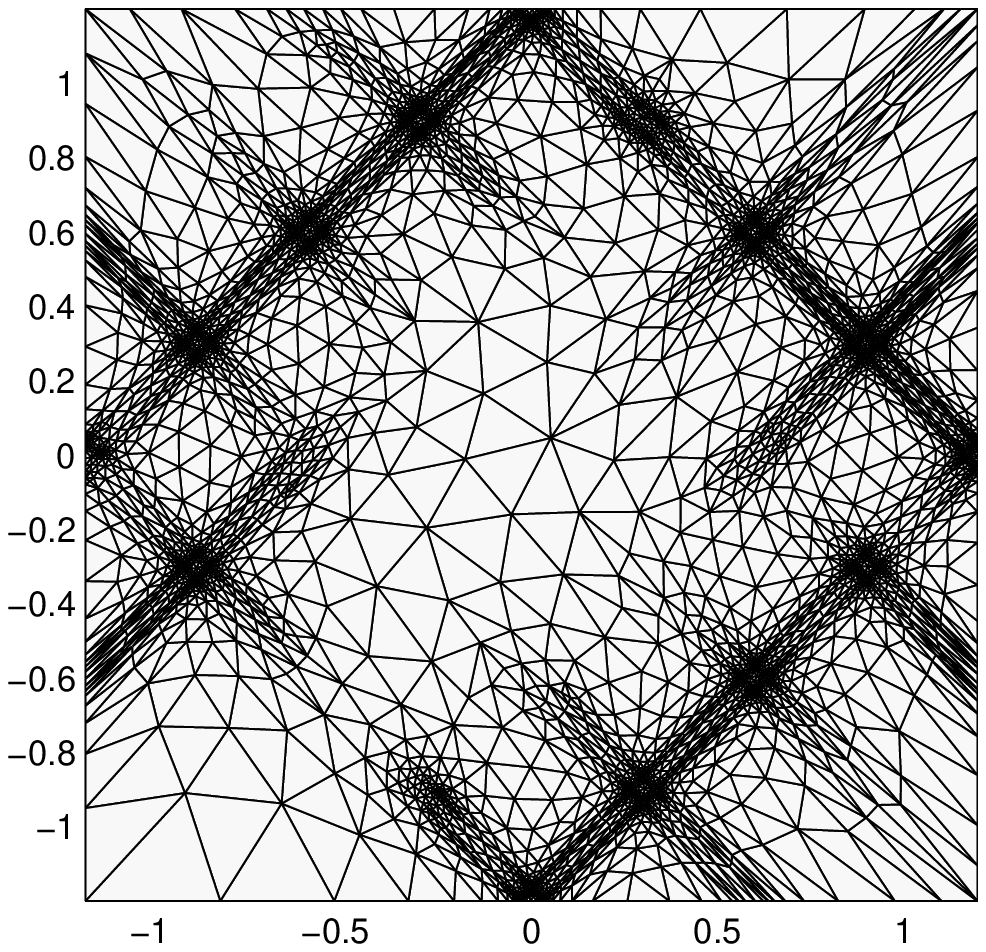}}
\hspace{-1.2cm}
\subfigure[]{\includegraphics[width=6cm]{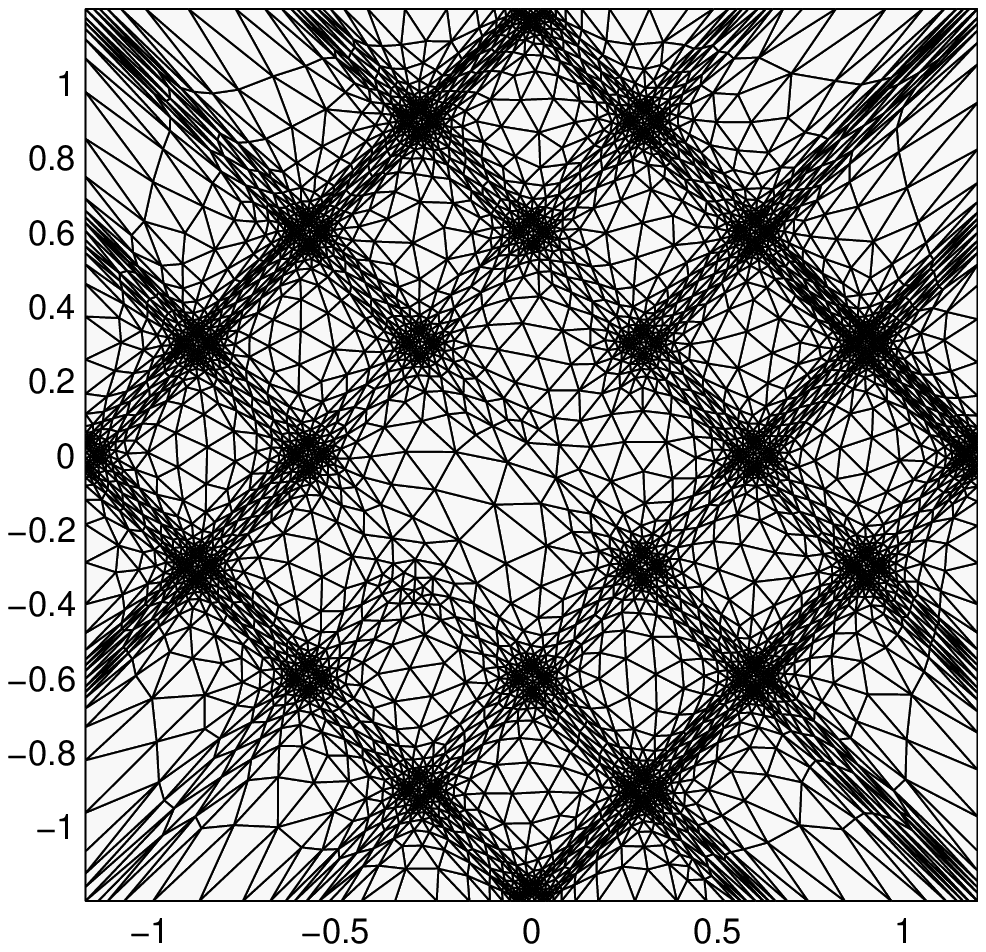}}

\hspace{-1.1cm}
\subfigure[]{\includegraphics[width=6cm]{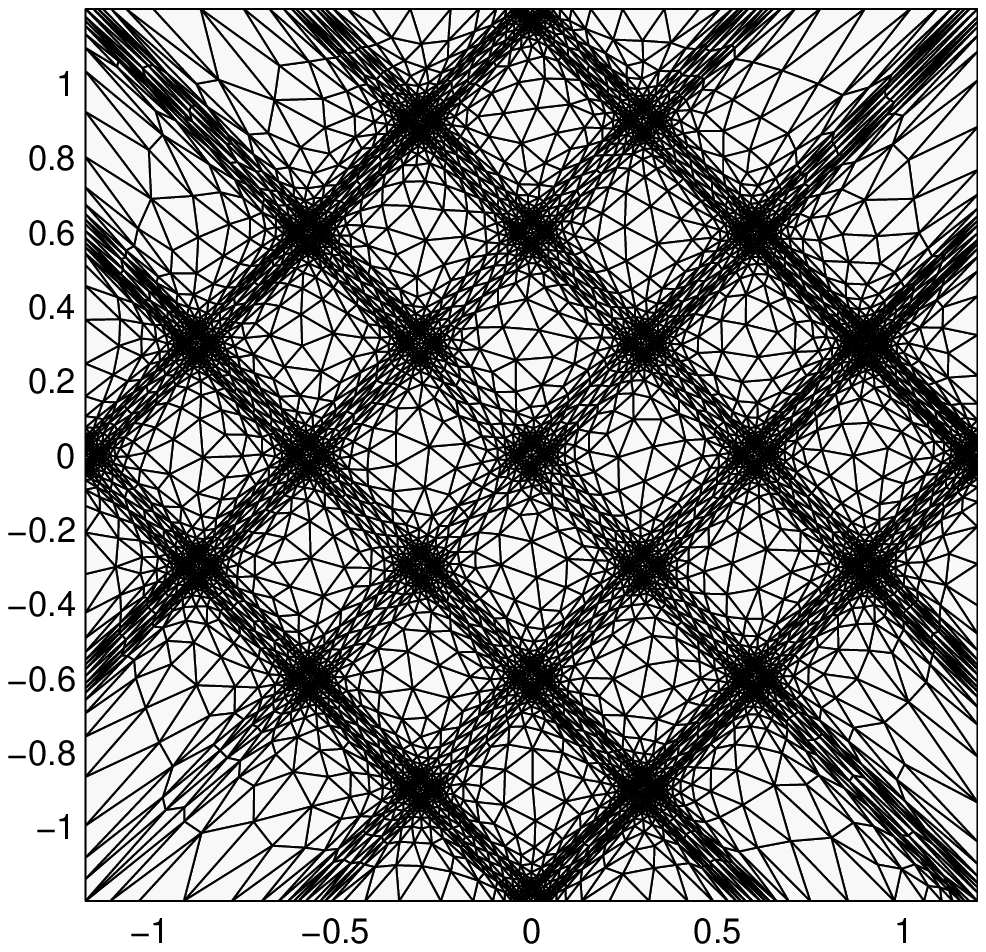}}

\caption{Example 3: Meshes generated by the metric tensors
$\mathcal{M}_{1,2}$ after (a) 5 step, (b) 10 step, (c) 15 step, and
 $\mathcal{M}_{1,\infty}$ after (d) 5 step, (e) 10 step, (f)
15 step, (g) 20 step.}\label{Fig_meshes_Tenlines_W1p}
\end{figure}

\begin{figure}
\hspace{-1cm}
\subfigure[]{\includegraphics[width=8cm]{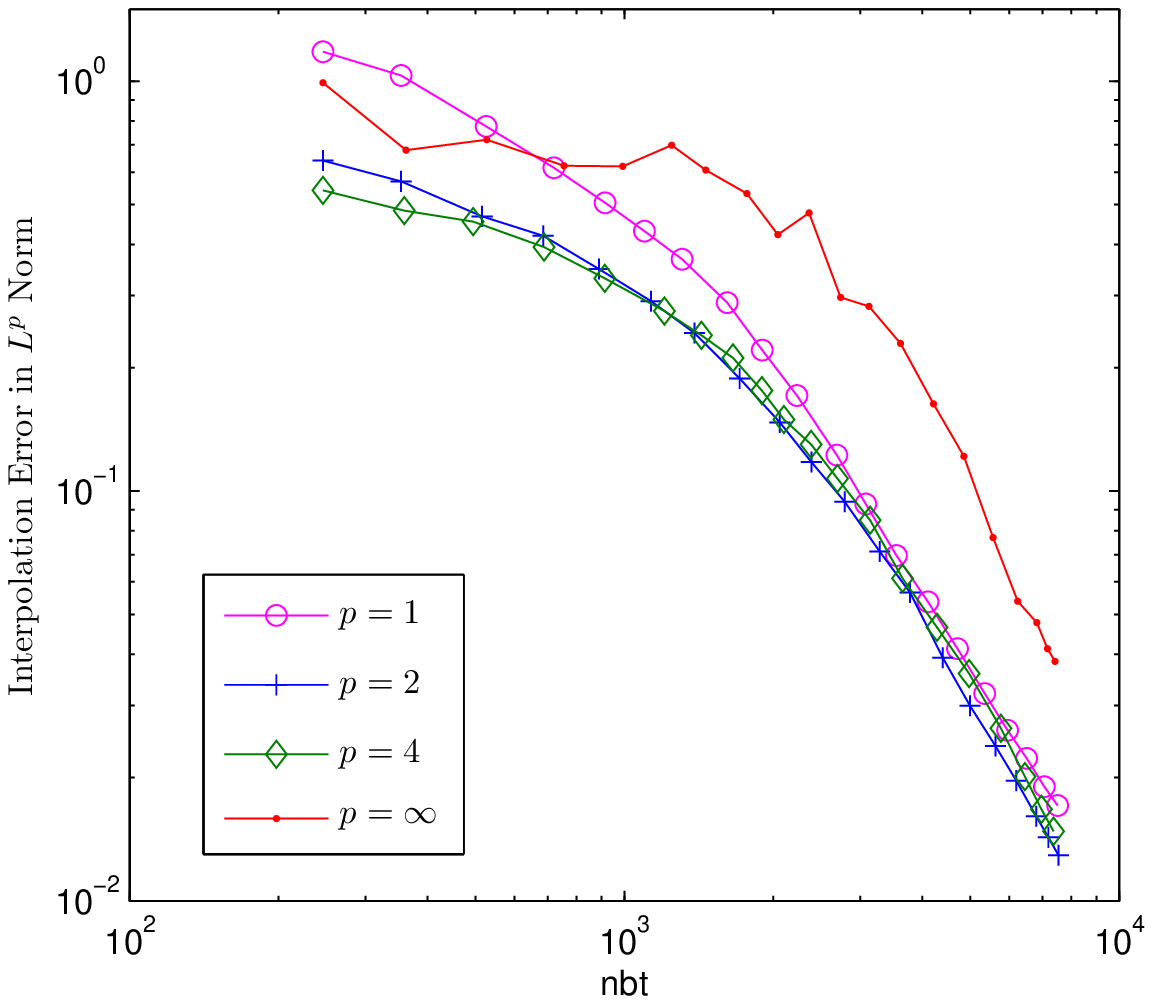}}
\subfigure[]{\includegraphics[width=8cm]{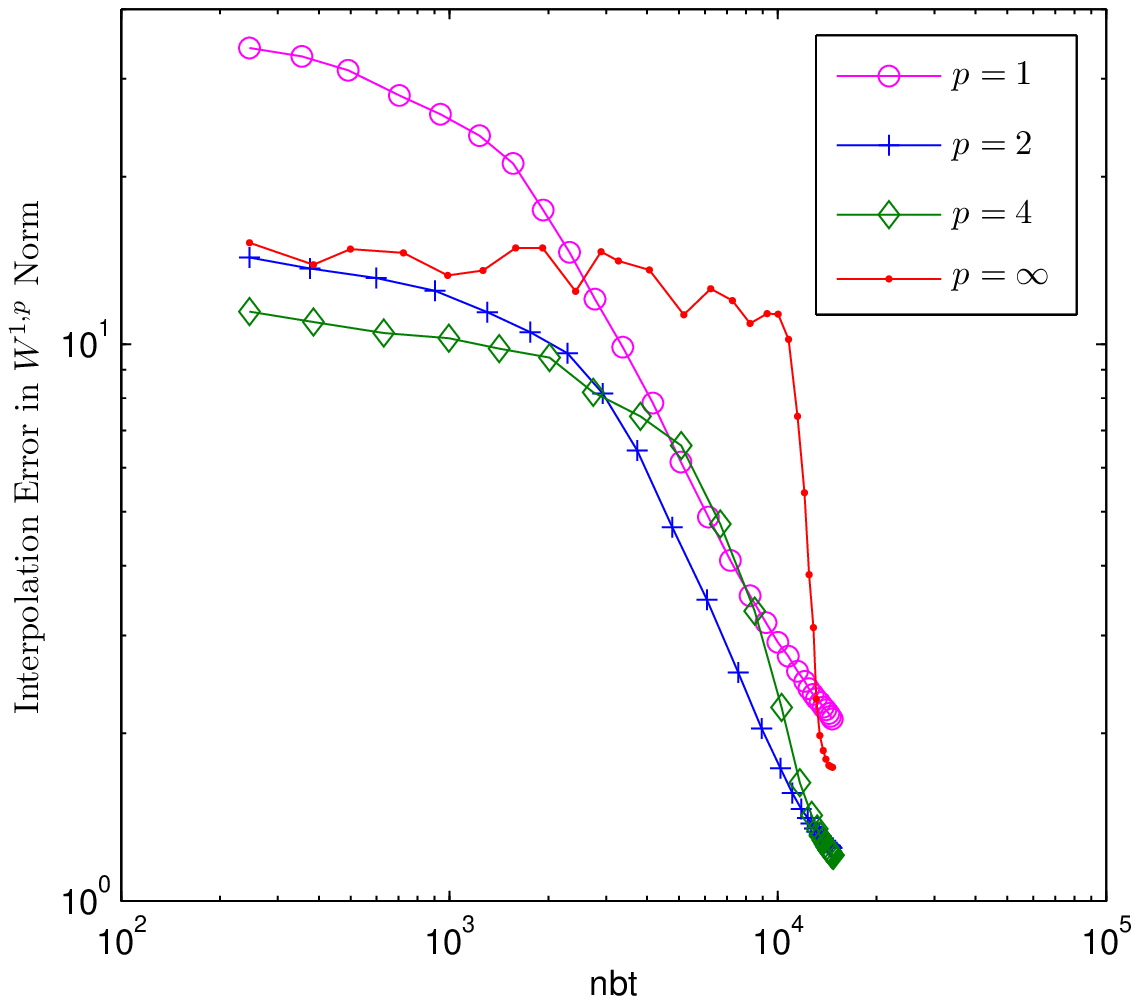}}
\caption{Example 3:  Convergence history versus number of triangles
by the metric tensor $\mathcal{M}_{0,p}$ and $\mathcal{M}_{1,p}$ for
$p=1, 2, 4, \infty$.}\label{Fig_loops_Tenlines}
\end{figure}

\section{Conclusions}
In the previous sections we have developed a uniform strategy to
derive metric tensors in two spatial dimension for interpolation
errors and their gradients in $L^p$ norm. The metric tensor
$\mathcal{M}_{0,p}^{n}$ for the $L^p$ norm of the interpolation
error is similar to some existing methods. However, the metric
tensor $\mathcal{M}_{1,p}^{n}$ is essentially different with those
metric tensors existed. There is a fine distinction between the new
metric tensor $\mathcal{M}_{1,p}^{n}$ and $\mathcal{M}_{1,p}^{h}$
proposed by Huang and Russell\cite{HuangRussell} that the former
refers to $\mbox{tr}(\mathcal{H})$ and the latter involves
$\|\mathcal{H}\|$. In some cases, the two metric tensors are pretty
much alike. However, when dealing with the complex problems, e.g.,
with multiple layers, the effect of the former is superior to the
latter for mesh generation. Numerical results show that the
corresponding convergent rates are always almost optimal.

\end{document}